\documentclass[11pt, a4paper]{article}
\usepackage{amssymb,amsmath,amsthm}
\usepackage{fancyhdr}
\usepackage[british]{babel}
\usepackage{enumitem}
\usepackage{lineno}
\usepackage{tikz}
\usepackage{todonotes}

\usepackage{hyperref}
\usepackage[margin=2.5cm]{geometry}
\newcounter{propcounter}
\newtheorem{theorem}{Theorem}[section]
\newtheorem{prop}[theorem]{Proposition}
\newtheorem{lemma}[theorem]{Lemma}
\newtheorem{cor}[theorem]{Corollary}

\theoremstyle{definition}

\newtheorem{fac}[theorem]{Fact}

\newtheorem{conj}[theorem]{Conjecture}

\newtheorem{defn}[theorem]{Definition}
\newtheorem{coro}[theorem]{Corollary}

\newtheorem{claim}[theorem]{Claim}

\newenvironment{pr}{{\noindent \it Proof:}}{\hfill $\blacksquare$\par}

\newcommand{\eps}{\varepsilon}
\newcommand{\<}{\subseteq}

\newcommand{\de}{\delta}
\newcommand{\be}{\beta}

\newcommand{\mf}[1]{\mathbf{#1}}
\newcommand{\overf}[1]{\overleftarrow{\mathbf{#1}}}

\title{On powers of Hamilton cycles in Ramsey-Tur\'{a}n Theory}

\author{
	Ming Chen\thanks{School of Mathematics and Statistics, Jiangsu Normal University, Xuzhou, China. Email: {\tt chenming314@gmail.com}.}
	\and
	Jie Han\thanks{School of Mathematics and Statistics and Center for Applied Mathematics, Beijing Institute of Technology, Beijing, China. Email: {\tt han.jie@bit.edu.cn}.}
	\and
    Yantao Tang\thanks{Zhongtai Securities Institute for Financial Studies, Shandong University, Jinan, China. Email: {\tt yttang@mail.sdu.edu.cn}.}
    \and
	Donglei Yang\thanks{Data Science Institute, Shandong University, Jinan, China, Email: {\tt dlyang@sdu.edu.cn}. D.Y. is supported by the China Postdoctoral Science Foundation (2021T140413) and Natural Science Foundation of China (12101365).}
}

\begin{document}
\maketitle
\begin{abstract}

We prove that for $r\in \mathbb{N}$ with $r\geq 2$ and $\mu>0$, there exist $\alpha>0$ and $n_{0}$ such that for every $n\geq n_{0}$, every $n$-vertex graph $G$ with $\delta(G)\geq \left(1-\frac{1}{r}+\mu\right)n$ and $\alpha(G)\leq \alpha n$ contains an $r$-th power of a Hamilton cycle.
We also show that the minimum degree condition is asymptotically sharp for $r=2, 3$ and the $r=2$ case was recently conjectured by Staden and Treglown.

\end{abstract}

\section{Introduction}\label{sec1}
A fundamental topic in graph theory is that of finding conditions under which a graph is Hamiltonian.
The classical result of Dirac \cite{Dirac1952} states that any graph on $n\geq 3$ vertices with minimum degree at least $\frac{n}{2}$ contains a Hamilton cycle.
The \emph{$r$-th power of a graph $H$} is the graph obtained from $H$ by joining every pair of vertices with distance at most $r$ in $H$.
As a natural generalization of Dirac's theorem, the P\'{o}sa--Seymour Conjecture received much attention, which predicts that an $n$-vertex graph $G$ satisfying $\delta(G)\geq \frac{r}{r+1}n$ contains an $r$-th power of a Hamilton cycle.
There had been many excellent results (see e.g. \cite{FGR1994},\cite{FK1995},\cite{HR1979}) until it was finally settled by Koml\'{o}s, S\'{a}rk\"{o}zy and Szemer\'{e}di \cite{KSGS1998} in 1997.

Note that the minimum degree condition is sharp as seen by near-balanced complete $(r+1)$-partite graphs.
As for many other problems in the area, this extremal example has the characteristic that it contains a large independence set.
There has thus been significant interest in seeking variants of classical results in extremal graph theory, where one now forbids the host graph from containing a large independent set.
Indeed, nearly 50 years ago, Erd\H{o}s, Hajnal, S\'{o}s and Szemer\'{e}di \cite{EHSS1983} initiated the study of the Tur\'{a}n problem under the additional assumption of small independence number.
Formally, given a graph $H$ and natural numbers $m, n\in \mathbb{N}$, the \emph{Ramsey--Tur\'{a}n number} $\textbf{RT}(n, H, m)$ is the maximum number of edges in an $n$-vertex $H$-free graph $G$ with $\alpha(G)\leq m$ and the \emph{Ramsey--Tur\'an density} of $K_r$ is defined as $\varrho(K_r):=\lim\limits_{\alpha\to0}\lim\limits_{n\to\infty}\frac{\textbf{RT}(n, K_r, \alpha n)}{\binom{n}{2}}$.
Erd\H{o}s--S\'{o}s \cite{EST1979} and Erd\H{o}s--Hajnal--S\'{o}s--Szemer\'{e}di \cite{EHSS1983} proved that
\begin{equation}\label{eq:rhoK}
\varrho(K_r) = \begin{cases}
1-\frac{2}{r-1}  \hfill\text{ if $r$ is odd,}\\
1-\frac{6}{3r-4} \hfill\text{ otherwise.}
\end{cases}
\end{equation}
For more literature on the Ramsey--Tur\'{a}n theory, we refer the readers to a comprehensive survey of Simonovits and S\'{o}s \cite{simsos2001}.

More recently, there has been interest in similar questions but where now one seeks for a spanning subgraph in an $n$-vertex graph with independence number $o(n)$ and large minimum degree.
Recent results mostly focus on clique-factors (and in general, $F$-factors).
In particular, Knierim and Su~\cite{MR4193066} showed that for fixed $\mu>0$, sufficiently small $\alpha>0$ and sufficiently large $n\in r\mathbb N$, an $n$-vertex graph $G$ with $\delta(G)\geq \left(1-\frac{2}{r}+\mu\right)n$ and $\alpha(G)\leq \alpha n$ contains a $K_{r}$-factor.
The $r=3$ case was obtained earlier by Balogh, Molla and Sharifzadeh~\cite{MR3570984}, who initiated this line of research.
Nenadov and Pehova~\cite{MR4080942} also proposed similar problems for sublinear $\ell$-independence number for $\ell\ge 2$.
%
%
For more problems and results on this topic, we refer to \cite{MR3570984,CHKWY,HHYW,HMWY2021,MR4193066,MR4080942,MR4170632}.


Moving attention to \emph{connected} spanning subgraphs (in contrast to $F$-factors), one quickly observes that we cannot reduce the minimum degree condition in Dirac's theorem (``$n/2$'') significantly, as seen by the graph formed by two disjoint cliques of almost equal size.
For a next step, Staden and Treglown \cite{MR4170632} conjectured that this minimum degree condition essentially guarantees a square of Hamilton cycle.

\begin{conj}[\cite{MR4170632}]\label{conj3}
For every $\mu>0$, there exist $\alpha>0$ and $n_{0}\in \mathbb{N}$ such that the following holds.
For every $n$-vertex graph $G$ with $n\geq n_{0}$, if $\delta(G)\geq \left(\frac{1}{2}+\mu\right)n$ and $\alpha(G)\leq \alpha n$, then $G$ contains a square of a Hamilton cycle.
\end{conj}

\subsection{Main results and Lower bound constructions}
Our main result is the following, which in particular resolves Conjecture \ref{conj3}. 

\begin{theorem}[]\label{thm2}
Given $\mu>0$ and $r\in \mathbb{N}$ with $r\geq 2$, there exists $\alpha>0$ such that the following holds for sufficiently large $n$.
Let $G$ be an $n$-vertex graph with $\delta(G)\geq \left(1-\frac{1}{r}+\mu\right)n$ and $\alpha(G)\leq \alpha n$.
Then $G$ contains an $r$-th power of a Hamilton cycle.
\end{theorem}

The minimum degree condition in Conjecture \ref{conj3} is asymptotically best possible, e.g. by considering a union of two disjoint cliques of almost equal size.
Moreover, we cannot expect a result where $\mu$ does not depend on $\alpha$ (at least for r=2), as seen by the following proposition.
\begin{prop}[]\label{prop1.2}
Given $\alpha>0$ and $n\in \mathbb{N}$, there exists an $n$-vertex graph $G$ with $\delta(G)\geq (\frac{1}{2}+\frac{\alpha}{2})n-1$ and $\alpha(G)\leq \alpha n$ such that $G$ has no square of a Hamilton cycle.
\end{prop}

%
%
For general lower bound on the minimum degree condition, we prove the following proposition by the \emph{connecting barrier}, which matches the minimum degree condition in Theorem~\ref{thm2} for the case $r=3$.

\begin{prop}[]\label{prop1.3}
Given $\mu, \alpha>0$ and $r\in \mathbb{N}$ with $r\geq 3$, the following holds for sufficiently large $n$. There exists an $n$-vertex graph $G$ with
$\delta(G)\geq \frac{2-\varrho(K_r)}{3-2\varrho(K_r)}n-\mu n$ and $\alpha(G)\leq \alpha n$ such that $G$ has no $r$-th power of a Hamilton cycle.
\end{prop}

The proofs of Proposition \ref{prop1.2} and Proposition \ref{prop1.3} are given in Subsection \ref{sec2}.
Note that $\varrho(K_3)=0$ and thus Proposition~\ref{prop1.3} complements Theorem~\ref{thm2} on the minimum degree condition for $r=3$.
For general $r$, note that by~\eqref{eq:rhoK}, we have
\[
\frac{2-\varrho(K_r)}{3-2\varrho(K_r)} = 1 - \frac{2}{r + c} > 1- \frac2{r+1},
\]
where $c=3$ if $r$ is odd and $c=8/3$ otherwise.
By the aforementioned result of Knierim and Su~\cite{MR4193066}, this shows a clear separation between the minimum degree thresholds for the $K_{r+1}$-factor problem and the $r$-th power of a Hamilton cycle problem in host graphs with sublinear independence number.
This is in contrast to the problems in general host graphs, where the two problems share the same minimum degree threshold.



We suspect that the ``connecting barrier'' in Proposition \ref{prop1.3} indeed discloses the best possible minimum degree condition forcing an $r$-th power of a Hamilton cycle for every $r\geq 2$ and propose the following conjecture.

\begin{conj}\label{conj4}
Given $\mu>0$ and $r\geq 4$, there exists $\alpha>0$ such that the following holds for sufficiently large $n$.
Let $G$ be an $n$-vertex graph with $\delta(G)\geq \frac{2-\varrho(K_r)}{3-2\varrho(K_r)}n+\mu n$ and $\alpha(G)\leq \alpha n$.
Then $G$ contains an $r$-th power of a Hamilton cycle.
\end{conj}

The smallest open case is $r=4$ where the minimum degree assumption in Theorem \ref{thm2} is roughly $\frac{3}{4}n$, while Proposition \ref{prop1.3} only provides a lower bound roughly $\frac{7}{10}n$.

\subsection{Connecting barrier}\label{sec2}
To obtain an $r$-th power of a Hamilton cycle, an essential condition is that every two vertices can be connected via an $r$-path.
Motivated by this, we give two constructions for the proofs of Proposition \ref{prop1.2} and Proposition \ref{prop1.3}, in which the connecting property collapses.

In the proofs, we apply a degree form of Ramsey--Tur\'{a}n number which was recently studied in \cite{CHKWY}, and their main result implies that every extremal construction for the function $\textbf{RT}(n, K_{r}, o(n))$ can be made almost regular.

\begin{coro}[\cite{CHKWY}, Proposition~1.5]\label{cor1}
Given $r \in \mathbb{N}$ with $r \ge 3$ and constants $\mu,\alpha>0$, the following holds for all sufficiently large $n\in \mathbb{N}$.
There is an $n$-vertex graph $G$ with $\delta(G)\ge \varrho(K_r) n-\mu n $ and $\alpha(G) \le \alpha n$ such that it does not contain any copy of $K_{r}$.
\end{coro}
The proofs of Proposition \ref{prop1.2} and Proposition \ref{prop1.3} go as follows.
We build a graph $G$ with $V(G)=V_{1}\cup V_{2}\cup V_{3}$.
Let $G[V_{2}, V_{1}]$ and $G[V_{2}, V_{3}]$ be two complete bipartite graphs, $E(G[V_{1}, V_{3}])=\emptyset$ and $G[V_{2}]$ be a $K_{r}$-free graph with large minimum degree.
We shall prove that there is no $r$-th power of a Hamilton cycle in $G$.
Then we optimize the size of $V_{i}$ for every $i\in [3]$ in order to maximize the minimum degree of $G$.

\begin{proof}[Proof of Proposition \ref{prop1.2}]
Given $\alpha>0$ and $n\in \mathbb{N}$, let $G$ be an $n$-vertex graph with $V(G)=V_{1}\cup V_{2}\cup V_{3}$, $|V_{2}|=\alpha n$ and $|V_{1}|=|V_{3}|=\frac{1-\alpha}{2}n$.
Let $E(G[V_{2}])=\emptyset$, $G[V_{i}]$ be a complete graph and $G[V_{2}, V_{i}]$ be a complete bipartite graph for every $i\in\{1, 3\}$.
It holds that $\delta(G)\geq(\frac{1}{2}+\frac{\alpha}{2})n-1$ and $\alpha (G)=\alpha n$.
Suppose for a contradiction that there exists a square of a Hamilton cycle in $G$, say $C$.
Let $C_{u, v}$ be a shortest $2$-power of a path between $u$ and $v$ in $C$.
We choose two vertices $u\in V_{1}$ and $v\in V_{3}$ such that $|C_{u, v}|$ is minimum among all such pairs in $V_1\times V_3$.
Then as $uv\notin E(G)$, it holds that $V(C_{u, v})\backslash \{u, v\}\< V_{2}$ and $|V(C_{u, v})\cap V_{2}|\geq 2$.
Since $E(G[V_{2}])=\emptyset$, this is a contradiction.
\end{proof}

\begin{proof}[Proof of Proposition \ref{prop1.3}]Given $\alpha,\mu>0$ and $r\in \mathbb{N}$ with $r\geq 3$, we choose $\frac{1}{n}\ll \alpha,\mu$.
Let $G$ be an $n$-vertex graph with $V(G)=V_{1}\cup V_{2}\cup V_{3}$, $|V_{1}|=|V_{3}|=\frac{1-\varrho(K_r)}{3-2\varrho(K_r)}n$ and $|V_{2}|=\frac{n}{3-2\varrho(K_r)}$.
Let $G[V_{i}]$ be a complete graph and $G[V_{2}, V_{i}]$ be a complete bipartite graph for every $i\in\{1, 3\}$.
Let $G[V_{2}]$ be a $K_{r}$-free subgraph with $\alpha(G[V_{2}])\leq \alpha n$ and $\delta(G[V_{2}])\geq \varrho(K_r)|V_2|-\mu n$ given by Corollary~\ref{cor1}.
For every vertex $v\in V_{1}$, it holds that $d(v)=n-|V_{3}|-1=\frac{2-\varrho(K_r)}{3-2\varrho(K_r)}n-1$.
For every vertex $v\in V_3$, it holds that $d(v)=n-|V_{1}|-1=\frac{2-\varrho(K_r)}{3-2\varrho(K_r)}n-1$.
For every vertex $v\in V_{2}$, it holds that
\begin{align*}
  d(v) & \geq |V_{1}|+|V_{3}|+\delta(G[V_{2}]) \\
   & \geq n-|V_{2}|+\varrho(K_r)|V_2|-\mu n
    = \frac{2-\varrho(K_r)}{3-2\varrho(K_r)}n-\mu n.
\end{align*}
Now $G$ has $\delta(G)\geq \frac{2-\varrho(K_r)}{3-2\varrho(K_r)}n-\mu n$ and $\alpha(G)\leq \alpha n$.
Suppose for a contradiction that there exists an $r$-th power of a Hamilton cycle in $G$, say $C$.
Let $C_{u, v}$ be a shortest $r$-path between $u$ and $v$ in $C$.
We choose two vertices $u\in V_{1}$ and $v\in V_{3}$ such that $|C_{u, v}|$ is minimum among all such pairs in $V_1\times V_3$.
Then as $uv\notin E(G)$, we obtain that $V(C_{u, v})\backslash \{u, v\}\< V_{2}$ and $|V(C_{u, v})\cap V_2|\geq r$, which forces a copy of $K_{r}$ in $G[V_2]$, a contradiction.
\end{proof}

\subsection{Proof strategy}
Our proof makes use of the absorption method and builds on the techniques developed in \cite{RM2014}.
The absorption method was introduced by R\"{o}dl, Ruci\'{n}ski and Szemer\'{e}di about a decade ago in \cite{MR2500161}.
Since then, it has turned out to be an important tool for studying the existence of spanning structures in graphs, digraphs and hypergraphs.
%
%
Under this framework, we find the desired $r$-th power of a Hamilton cycle by four lemmas: the Absorber Lemma (Lemma \ref{lem2.1}), the Reservoir Lemma (Lemma \ref{lem2.2}), the Almost Path-cover Lemma (Lemma \ref{almost}) and the Connecting Lemma (Lemma \ref{connect}).

The main difficulties lie in proving a connecting lemma.
For every two copies of $K_r$, say $X$ and $Y$,
Koml\'{o}s, S\'{a}k\"{o}zy and Szemer\'{e}di \cite{KSGS1998} also proved such a connecting lemma: they first use the regularity method and transfer the problem to the \emph{reduced graph}, and then use the \emph{cascade} technique to do the connection via an $r$-th power of a walk.
Then the regularity method can find an embedding of an $r$-path that connects $X$ and $Y$.
Under our weaker minimum degree condition, using their method we can make such a connection in the reduced graph via an \emph{$(r-1)$-power} of a path.
Then some dedicate embedding using the regularity and the assumption of sublinear independence number\footnote{For now, one can simply think of that every cluster in the regularity partition contains many edges inside.} can help us to achieve the connection by an $r$-path of constant length.
To reduce the technicality, we first present an easy way to connect (in the reduced graph) via a short sequence of copies of $K_{r+1}$, where consecutive pairs of them share $r-1$ vertices; secondly, to ease the embedding process, we introduce the notion of \emph{good walk} (see Definition~\ref{def6.0}) and use it to model the way how an $r$-path is embedded into a sequence of clusters.

For the absorbing path, we first choose a random set $A$ of vertices and find an $r$-path $P$ where one can free any subset of vertices of $A$.
Then the vertices of $A$ can be used to help on connecting $r$-paths in later steps.

Here is a sketch the proof of Theorem~\ref{thm2}.
We apply the Reservoir Lemma to obtain a random set $A$.
Then we apply the Absorber Lemma, and find a collection of vertex-disjoint absorbers in $G-A$, one for each vertex of $A$.
By applying the Connecting Lemma, we connect all these $|A|$ absorbers one by one by using the vertices in $V(G)\backslash A$, and obtain an $r$-path, say $P_1$, which contains all these $|A|$ absorbers and every vertex in $A$.
Note that the $r$-path $P_1$ is an absorbing path in the sense that we can remove any subset of vertices of $A$ while keeping the ends of the $r$-path unchanged.
In $G-V(P_1)$, we apply the Almost Path-cover Lemma, and obtain a family of vertex-disjoint $r$-paths, say $\{P_2, P_3, \dots, P_\ell\}$ for some $\ell = o(n)$, and let $R:=V(G)\setminus \bigcup^{\ell}_{i=1}V(P_i)$.
By the choice of $A$, every vertex in $G$ has many neighbors in $A$ and thus many absorbers in $A$.
We then choose a collection of vertex-disjoint short $r$-paths that contain the vertices of $R$ with other vertices in $A$ -- this can be achieved by just choosing vertex-disjoint absorbers for vertices of $R$.
By applying the Connecting Lemma, we connect all these $|R|$ (short) $r$-paths and all these $\ell$ (long) $r$-paths one by one to an $r$-th power of a cycle by using the vertices in $A$.
Since all the vertices used are from $A$ (and thus from $P_1$), we obtain an $r$-th power of a Hamilton cycle in $G$.

%

\subsection{Basic notation}
In this subsection, we include some notation used throughout the paper. For a graph $G:= G(V, E)$, we write $|G|=|V(G)|$ and $e(G)=|E(G)|$.
For $U\< V(G)$, $G[U]$ denotes the induced graph of $G$ on $U$.
Let $G-U:=G[V(G)\backslash U]$.
For two subsets $A, B\< V(G)$, we use $E(A, B)$ to denote the set of edges with one endpoint in $A$ and the other in $B$.
Given $A\subseteq V(G)$ and $v\in V(G)$, $N_{A}(v):=N(v)\cap A$, and $d_{A}(v):=|N_{A}(v)|$.
When $A=V(G)$, we drop the subscript and simply write $d(v)$.
For $A\subseteq V(G)$, $N(A):=\bigcap_{v\in A}N(v)$.
For any integers $a\leq b$, $[a, b]:=\{i\in \mathbb{Z}: a\leq i\leq b\}$ and $[a]:= [1, a]$.
Given $r\in \mathbb{N}$, we use $\mathbf{S}_{r}$ to denote the family of all permutations of $[r]$.
Given $\pi\in \mathbf{S}_{r}$ and $i\neq j\in [r]$, we use $\pi\circ(i, j)$ to denote the permutation in $\mathbf{S}_{r}$ which comes from $\pi$ by swapping the $i$-th and $j$-th elements in $\pi$.
Given a vertex set $V=\{v_{1}, \dots, v_{\ell}\}$, let $\mathbf{S}_{V}$ be the family of all permutations of the vertices $v_{1}, \dots, v_{\ell}$.

Given an $r$-path $P=v_{1}v_{2}\dots v_{\ell}$ for some $\ell\in \mathbb{N}$ and $\ell\geq r$, two \emph{ends} of $P$ are defined as the $r$-tuples $(v_{r}, v_{r-1}, \dots, v_{1})$ and $(v_{\ell-r+1}, \dots, v_{\ell-1}, v_{\ell})$.
Given two disjoint $r$-tuples of vertices $\mf{x}, \mf{y}$ each inducing a copy of $K_{r}$, say $\mf{x}=(x_{1}, \dots, x_{r})$ and $\mf{y}=(y_{1},\dots, y_{r})$, and an $r$-path $P=v_{1}\dots v_{\ell}$ for some $\ell\in \mathbb{N}$ that is vertex disjoint from $\mf{x}$ and $\mf{y}$, we say $P$ \emph{connects} $\mf{x}$ and $\mf{y}$ if $v_{i}\in N(x_{i}, \dots, x_{r})$ and $y_{i}\in N(v_{\ell-r+i}, \dots, v_{\ell})$ for every $i\in [r]$.
Moreover, we use $\mf{x} P \mf{y}$ to denote the resulting $r$-path $x_{1}\dots x_{r}v_1v_2\ldots v_{\ell}y_{1}\dots y_{r}$.
Given two sequences $Q_{1}:=(S_1, \dots, S_{r})$ and $Q_{2}:=(T_{1}, \dots, T_{\ell})$, let $Q_{1}Q_2:=(S_1, \dots, S_{r}, T_{1}, \dots, T_{\ell})$.

When we write $\beta\ll \gamma$, we always mean that $\beta, \gamma$ are constants in $(0, 1)$, and there exists $\beta_{0}=\beta_{0}(\gamma)$ such that the subsequent statements hold for all $0<\beta\leq \beta_{0}$.
Hierarchies of other lengths are defined analogously.

The rest of the paper is organized as follows.
In Section \ref{sec3} we will give the proof of Theorem \ref{thm2}.
The proofs of the Absorber Lemma, the Reservoir Lemma and the Almost path-cover Lemma are in Subsection \ref{sec4}.
In Section \ref{sec6} we will introduce the regularity lemma and prove the Connecting Lemma which will comprise the majority of the paper.

\section{Proof of Theorem \ref{thm2}}\label{sec3}
In this section, we introduce the crucial lemmas used to prove our main result.
We explain how they work together to derive the proof of Theorem \ref{thm2}.
The proofs of these lemmas are presented in full detail in Subsection \ref{sec4} and Section \ref{sec6}, respectively.


\subsection{Main tools}\label{sec3.1}
Before presenting the statements of these crucial lemmas, we need to introduce one more notion.

\begin{defn}[Absorber]
Given $v\in V(G)$, we say that an $r$-path $P$ of order $2r$ is an \emph{absorber} for $v$ if $V(P)\cup\{v\}$ induces an $r$-path of order $2r+1$ (in $G$) which shares the same ends with $P$.
For $v\in V(G)$ and $U\subseteq V(G)$, let $\mathcal{L}_{U}(v)$ be a maximum set of vertex-disjoint absorbers of $v$ in $G[U]$. We often omit the subscript $U$ if the graph is clear from the context.
\end{defn}


We are actually able to prove a slightly stronger result by replacing $1-\frac{1}{r}$ in Theorem \ref{thm2} with $1-\frac{1}{f(r)}$ for some $0<f(r)\leq r$.

\begin{lemma}[Absorber Lemma]\label{lem2.1}
Given $r\in \mathbb{N}$ and $\mu>0$, let $f(r):=\frac{r}{2}+1$ if $r\in 2\mathbb{N}$, and $f(r):=\frac{r}{2}+\frac{5}{6}$ if $r\in 2\mathbb{N}+1$.
There exists $\alpha>0$ such that the following holds for sufficiently large $n$.
Let $G$ be an $n$-vertex graph with $\delta(G)\geq \left(1-\frac{1}{f(r)}+\mu\right)n$, $\alpha(G)\leq \alpha n$, and $W\< V(G)$ with $|W|\leq\frac{\mu}{2}n$.
Then every $v\in V(G)$ has at least one absorber in $G[N(v)\backslash W]$.
\end{lemma}

For convenience, we will use the following immediate corollary of Lemma \ref{lem2.1}.

\begin{cor}\label{coro2.1}
Given $r\in \mathbb{N}$ and $\mu>0$, let $f(r):=\frac{r}{2}+1$ if $r\in 2\mathbb{N}$, and $f(r):=\frac{r}{2}+\frac{5}{6}$ if $r\in 2\mathbb{N}+1$.
There exists $\alpha>0$ such that the following holds for sufficiently large $n$.
Let $G$ be an $n$-vertex graph with $\delta(G)\geq \left(1-\frac{1}{f(r)}+\mu\right)n$ and $\alpha(G)\leq \alpha n$.
Then $|\mathcal{L}(v)|\geq \frac{\mu}{4r}n$ for every $v\in V(G)$.
\end{cor}


As mentioned earlier, we need to reserve a vertex subset in the original graph to connect a small number of $r$-paths.
The following lemma provides such a vertex subset.

\begin{lemma}[Reservoir Lemma]\label{lem2.2}
Given $c, \mu, \eta, \gamma>0$, there exists $\zeta>0$ such that the following holds for sufficiently large $n$.
Let $G$ be an $n$-vertex graph with $\delta(G)\geq \left(c+\mu\right)n$, and $|\mathcal{L}(v)|\geq \eta n$ for every $v\in V(G)$.
Then there exists a vertex subset $A\< V(G)$ with $2r\zeta n\leq |A|\leq \gamma n$ such that for every $v\in V(G)$, it holds that $|N_A(v)|\geq \left(c+\frac{\mu}{2}\right)|A|$ and $|\mathcal{L}_{A}(v)|\geq \zeta n$.
\end{lemma}


The following lemma provides a family of vertex-disjoint $r$-paths that cover almost all vertices in $G$.
We are actually able to prove a slightly stronger result by replacing $1-\frac{1}{r}$ in Theorem \ref{thm2} with $1-\frac{1}{g(r)}$ for some $0<g(r)\leq r$.

\begin{lemma}[Almost path-cover Lemma]\label{almost}
Given $\mu, \delta>0$ and $r, \ell\in \mathbb{N}$, let $g(r):=\lfloor\frac{r}{2}\rfloor+1$.
There exists $\alpha>0$ such that the following holds for sufficiently large $n$.
Let $G$ be an $n$-vertex graph with $\delta(G)\geq \left(1-\frac{1}{g(r)}+\mu\right)n$ and $\alpha(G)\leq \alpha n$.
Then $G$ contains a family of vertex-disjoint $r$-paths with the same length $\ell$, which covers all but at most $\delta n$ vertices in $G$.
\end{lemma}

We also require the following result to connect two vertex-disjoint $r$-paths. This lemma is the only part of the proof of Theorem \ref{thm2} that requires $\delta(G)\geq (1-\frac{1}{r}+\mu)n$ (elsewhere, $\delta(G)\geq (1-\frac{2}{r+2}+\mu)n$ suffices).

\begin{lemma}[Connecting Lemma]\label{connect}
Given $r\in \mathbb{N}$ and $\mu>0$, there exists $\alpha>0$ such that the following holds for sufficiently large $n$.
Let $G$ be an $n$-vertex graph with $\delta(G)\geq \left(1-\frac{1}{r}+\mu\right)n$ and $\alpha(G)\leq \alpha n$.
For every two disjoint $r$-tuples of vertices, denoted as $\mf{x}$ and $\mf{y}$, each inducing a copy of $K_{r}$ in $G$, there exists an $r$-path in $G$ on at most $200r^{5}$ vertices, which connects $\mf{x}$ and $\mf{y}$.
\end{lemma}

More often, given an $r$-tuple $\mf{x}=(x_1,x_2,\ldots,x_r)$, we write $\overf{x}:=(x_r,x_{r-1},\ldots,x_1)$.
Now we are ready to prove Theorem \ref{thm2} using Corollary \ref{coro2.1}, Lemma \ref{lem2.2}, Lemma \ref{almost} and Lemma \ref{connect}.

\begin{proof} [Proof of Theorem \ref{thm2}]
Given $\mu>0$ and $r\in \mathbb{N}$, we choose
\begin{center}
$\frac{1}{n}\ll \alpha\ll\frac{1}{\ell}, \delta\ll\zeta\ll\gamma\ll\mu, \frac{1}{r}$.
\end{center}
Let $\eta=\frac{\mu}{4r}$, $G$ be an $n$-vertex graph with $\delta(G)\geq \left(1-\frac{1}{r}+\mu\right)n$ and $\alpha(G)\leq \alpha n$.
Applying Corollary \ref{coro2.1} to $G$, we get that $|\mathcal{L}(v)|\geq \frac{\mu}{4r}n$ for every vertex $v\in V(G)$.
By Lemma \ref{lem2.2}, we obtain a reservoir $A\< V(G)$ with
  \stepcounter{propcounter}
\begin{enumerate}[label = ({\bfseries \Alph{propcounter}\arabic{enumi}})]
       \item\label{p1} $2r\zeta n\leq|A|\leq \gamma n$;
       \item\label{p2} $|\mathcal{L}_{A}(v)|\geq \zeta n$ for every $v\in V(G)$;
       \item\label{p3} $|N_A(v)|\geq \left(1-\frac{1}{r}+\frac{\mu}{2}\right)|A|$ for every $v\in V(G)$.
\end{enumerate}
Since $\gamma\ll \mu, \frac{1}{r}$ and $|\mathcal{L}(v)|\geq \frac{\mu}{4r}n$ for every $v\in V(G)$, we can greedily choose an absorber for every $v\in A$ such that they are pairwise vertex disjoint from each other and from $A$.
Since $\gamma \ll \mu, \frac{1}{r}$ and $|A|\leq \gamma n$,
this is possible because during the process, the number of vertices that we need to avoid is at most $(2r+1)|A|\leq (2r+1)\gamma n< \frac{\mu}{4r}n\leq |\mathcal{L}(v)|$.
Let $\mathcal{F}_{1}:=\{H_{1}, H_{2}, \dots, H_{|A|}\}$ be the family of these $r$-paths, each of which is the $r$-path formed by a vertex in $A$ \emph{together with} \footnote{Each $H_i$ has the same ends as the corresponding absorber.} its absorber.
Let $\mf{x}_i$ and $\mf{y}_i$ be two ends of $H_i$ for every $i\in[|A|]$.
Then we shall connect all these $r$-paths $H_i$ in pairs into an $r$-path.
In every step of connecting $H_i$ and $H_{i+1}$, we shall find an $r$-path of length at most $200r^5$ connecting $\mf{y}_{i}$ and $\overleftarrow{\textbf{x}_{i+1}}$, which is vertex disjoint from all previous connections and $V(\mathcal{F}_1)$.
Since $\frac{1}{n}\ll \gamma\ll\mu, \frac{1}{r}$ and $|A|\leq \gamma n$, such a family of $|A|-1$ vertex-disjoint $r$-paths can be iteratively obtained by repeatedly applying Lemma~\ref{connect} to $G-W$ with $W$ being the set of vertices that are in $ V(\mathcal{F}_1)$ or used in previous connections.
After all these connections, we obtain an $r$-path, say $P_{1}$, which contains all the members of $\mathcal{F}_{1}$, and thus $A$.
Since $\gamma \ll \mu, \frac{1}{r}$ and $|A|\leq \gamma n$, $|P_{1}|\leq 200r^{5}(|A|-1)+(2r+1)|A|<(3r+200r^{5})|A|\leq \frac{\mu}{4} n$.
Observer that for every subset $A'\subseteq A$ there is an $r$-path on $V(P_1) \setminus A'$ with the same ends as $P_1$.

Note that we have
$\delta(G-V(P_{1})) \geq \left(1-\frac{1}{r}+\mu\right)n-\frac{\mu}{4} n\geq \left(1-\frac{1}{r}+\frac{3}{4}\mu\right)n$.
Applying Lemma \ref{almost} to $G-V(P_{1})$, we obtain a collection of vertex-disjoint $r$-paths, each of length $\ell$, denoted as $P_{2}, \dots, P_{\lambda}$ for some integer $\lambda\leq\frac{n}{\ell}$.
Moreover, this collection covers all but a set $R$ of at most $\delta n$ vertices in $G-V(P_{1})$.

Recall that by \ref{p2}\ref{p3}, for every $v\in R$, we have $|\mathcal{L}_{A}(v)|\geq \zeta n$ and $|N_A(v)|\geq \left(1-\frac{1}{r}+\frac{\mu}{2}\right)|A|$.
As $\delta\ll\zeta, \frac{1}{r}$ and $|R|\le \delta n$, we can greedily pick an absorber in $G[A]$ for every $v\in R$ such that all such absorbers are pairwise vertex disjoint from each other.
This is possible since each time for $v\in R$ the number of absorbers of $v$ touched in previous steps is at most $2r|R|\leq 2r\delta n<\zeta n \leq|\mathcal{L}_{A}(v)|$.
Let $\mathcal{F}_{2}:=\{H_{|A|+1}, H_{|A|+2}, \dots, H_{|A|+|R|}\}$ be the family of $r$-paths each of which is formed by a vertex in $R$ together with its absorber.

Now we connect all these $r$-paths to an $r$-th power of a (Hamilton) cycle using the vertices of $A$.
In fact, since $\frac{1}{\ell}, \delta\ll\zeta\ll \mu, \frac{1}{r}$, $|A|\geq 2r\zeta n$, $|R|\le \delta n$ and $|N_A(v)|\ge \left(1-\frac{1}{r}+\frac{\mu}{2}\right)|A|$ for every $v\in V(G)$, this can be similarly done by iteratively applying Lemma~\ref{connect} to $G[A]-W$ with $W$ being the set of vertices in $V(\mathcal{F}_2)$ and those used in previous connections, and by the fact that
\begin{center}
$|W|\le (2r+1)|R|+200r^{5}(|R|+\lambda)\leq 300r^5(|R|+\lambda)\leq 300r^5(\delta n+\frac{n}{\ell})\leq \frac{\mu}{4}|A|$.
\end{center}
Let $A'\subseteq A$ be the set of all the vertices used for the connections.
By the observation of $P_1$ as above, we can obtain an $r$-path on $V(P_1)\setminus A'$ with the same ends as $P_1$.
This gives an $r$-th power of a Hamilton cycle of $G$.
\end{proof}

\subsection{Proof of the lemmas}\label{sec4}
In this subsection, we provide the (short) proofs for Lemma \ref{lem2.1}, Lemma \ref{lem2.2}, and Lemma~\ref{almost}.
The proof of Lemma \ref{lem2.1} utilizes some results on the Ramsey--Tur\'{a}n number.
We prove Lemma \ref{lem2.2} using some standard probabilistic arguments.
To prove Lemma \ref{almost}, we apply a recent result of Chen, Han, Wang and Yang in \cite{CHWY2022}.

To prove Lemma \ref{lem2.1}, we need the following notation.
We use $ar(H)$ to denote the \emph{vertex arboricity} of $H$ which is the least integer $r$ such that $V(H)$ can be partitioned into $r$ parts and each part induces a forest in $H$.
The corresponding partition of $V(H)$ is an \emph{acyclic partition} of $H$.
\begin{defn}[\cite{EHSS1983}]\label{def1.5}
The \emph{modified arboricity} $AR(H)$ of a graph $H$ is the least $\ell\in \mathbb{N}$ for which the following holds:
\begin{itemize}
  \item either $\ell$ is even and $ar(H)=\frac{\ell}{2}$;
  \item or $\ell$ is odd and there exists an independent set $V^{\ast}$ such that $ar(H-V^{\ast})\leq \frac{\ell-1}{2}$.
\end{itemize}
\end{defn}
\begin{theorem}[\cite{EHSS1983}]\label{thm0.3}
Let $h, \ell\in \mathbb{N}$ and $H$ be an $h$-vertex graph.
If $AR(H)\leq \ell$, then $\textbf{RT}(n, H, o(n))\leq \textbf{RT}(n, K_{\ell}, o(n))$.
\end{theorem}


\begin{prop}[]\label{prop3.3}
Given $r\geq 2$, let $H:=P^{r}_{2r}$.
Then $AR(H)=r+1$.
\end{prop}
\begin{proof}
Let $H$ be an $r$-path $P=v_{1}\dots v_{2r}$.
Since $K_{r+1}\<H$, we have $AR(H)\ge AR(K_{r+1})=r+1$.
Assume that $r$ is even, we obtain an acyclic partition of $H$ with $\frac{r}{2}+1$ parts
\[\{v_{1}, v_{2}, v_{r+2}, v_{r+3}\},\{v_{3}, v_{4}, v_{r+4}, v_{r+5}\}, \cdots, \{v_{r-1}, v_{r}, v_{2r}\},\{v_{r+1}\}.\]
Thus, by definition, we obtain that $AR(H)=r+1$.
Assume that $r$ is odd, we obtain an acyclic partition of $H$ with $\frac{r+1}{2}$ parts
\[\{v_{1}, v_{2}, v_{r+2}, v_{r+3}\}, \cdots, \{v_{r-2}, v_{r-1}, v_{2r-1}, v_{2r}\}, \{v_{r}, v_{r+1}\}.\]
Therefore, by definition, we obtain that $AR(H)=r+1$.
\end{proof}
In the next subsection, we prove Lemma \ref{lem2.1}.
\subsubsection{Absorbers: Proof of Lemma \ref{lem2.1}}\label{sec2.2.1}
Before presenting the proof, we first recall some of the results mentioned earlier regarding Ramsey--Tur\'{a}n theory.
In \cite{EST1979}, Erd\H{o}s and S\'{o}s proved that let $r\in 2\mathbb{N}+1$, $\textbf{RT}(n, K_{r}, o(n))=\frac{r-3}{2r-2}n^{2}+o(n^{2})$.
In \cite{EHSS1983}, Erd\H{o}s, Hajnal, S\'{o}s and Szemer\'{e}di proved that
let $r\in 2\mathbb{N}$, $\textbf{RT}(n, K_{r}, o(n))=\frac{3r-10}{6r-8}n^{2}+o(n^{2})$.

\begin{proof} [Proof of Lemma \ref {lem2.1}]
Given $r\in \mathbb{N}$ and $\mu>0$, we choose
$\frac{1}{n}\ll\alpha\ll \mu$.
Recall $f(r)=\frac{r}{2}+1$ if $r\in 2\mathbb{N}$, and $f(r)=\frac{r}{2}+\frac{5}{6}$ if $r\in 2\mathbb{N}+1$.
Let $f:=f(r)$ for every $r\in \mathbb{N}$, $G$ be an $n$-vertex graph with $\delta(G)\geq \left(1-\frac{1}{f}+\mu\right)n$, $W\< V(G)$ with $|W|\leq\frac{\mu}{2}n$ and $v\in V(G)$.
We shall find an absorber of $v$ in $G[N_G(v)\setminus W]$, that is, a copy of $P^{r}_{2r}$.
Write $V'=N_G(v)\setminus W$ and $n'=|V'|$.
Then $n'\ge \frac{f-1}{f}n$.

By Theorem \ref{thm0.3} and Proposition \ref{prop3.3}, it suffices to prove that $e(G[V'])\geq \textbf{RT}(n', K_{r+1}, 2\alpha n')$.
We have $\delta(G[V']) \ge n' - \frac{n}{f} + \mu n\ge n' - \frac{n'}{f-1}  + \mu n'$.
This implies that the density of $G[V']$ is at least $1-\frac{1}{f-1}+\mu$.
By (\ref{eq:rhoK}) and the definition of $f=f(r)$, we infer that $G[V']$ contains a copy of $K_{r+1}$.
This completes the proof.
\end{proof}
\subsubsection{Reservoir: Proof of Lemma \ref{lem2.2}}\label{sec2.2.2}
The proof of Lemma \ref{lem2.2} involves a standard probabilistic argument. In order to carry out this argument, we will need to make use of the following Chernoff bound for the binomial distribution (see Corollary 2.3 in \cite{JSLT2000}). Recall that the binomial random variable with parameters $(n, p)$ is the sum of $n$ independent Bernoulli variables, each taking value $1$ with probability $p$ and $0$ with probability $1-p$.

\begin{prop}[\cite{JSLT2000}]\label{prop6.1}
Suppose that $X$ has the binomial distribution and $0<a<\frac{3}{2}$. Then $\mathbb{P}[|X-\mathbb{E}[X]|\geq a\mathbb{E}[X]]\leq 2\exp(-a^{2}\mathbb{E}[X]/3)$.
\end{prop}

\begin{proof} [Proof of Lemma \ref {lem2.2}]
Given $c, \mu, \eta, \gamma>0$, we choose
$\frac{1}{n}\ll \eta, \mu, \gamma, c$ and let $c_{1}=c_{2}=\frac{\mu}{4}$.
Let $G$ be an $n$-vertex graph with $\delta(G)\geq \left(c+\mu\right)n$, and every $v\in V(G)$ has a set $\mathcal{L}(v)$ of at least $\eta n$ vertex-disjoint absorbers.
We first note that $c+\mu\leq 1$.
We choose a vertex subset $A\< V(G)$ by including every vertex independently at random with probability $p=\frac{\gamma}{2}$.
Notice that $\mathbb{E}[|A|]=np$, $\mathbb{E}[|\mathcal{L}_{A}(v)|]\ge\eta np^{2r}$ and $\mathbb{E}[d_A(v)]\ge\left(c+\mu\right)np$ for every $v\in V(G)$.
Then by Proposition \ref{prop6.1}, it holds that:
\begin{enumerate}
  \item $\mathbb{P}[|A|\geq(1+c_{1})\mathbb{E}[|A|]]\leq \mathbb{P}[||A|-\mathbb{E}[|A|]|\geq c_{1}\mathbb{E}[|A|]]\leq 2\exp\left(-c^{2}_{1}np/3\right)$;
  \item $\mathbb{P}[|\mathcal{L}_{A}(v)|\leq \frac{1}{2}\mathbb{E}[|\mathcal{L}_{A}(v)|]]\leq 2\exp\left(-\eta np^{2r}/12\right)$;
  \item $\mathbb{P}[d_A(v)\leq (1-c_{2})\mathbb{E}[d_A(v)]]\leq 2\exp\left(-c^{2}_{2}(c+\mu)np/3\right)$.
\end{enumerate}
Meanwhile, it holds that
\begin{center}
$\exp\left(-c^{2}_{1}np/3\right)+n\exp\left(-c^{2}_{2}(1-\frac{1}{r}+\mu)np/3\right)+n\exp\left(-\eta np^{2r}/12\right)\leq\frac{1}{6}$.
\end{center}
By the union bound, with probability at least $\frac{2}{3}$ the set $A$ satisfies the following properties:
  \[ |A| \leq (1+\frac{\mu}{4})\mathbb{E}[|A|]=(1+\frac{\mu}{4})np=(1+\frac{\mu}{4})\frac{\gamma}{2}n\leq \gamma n;\]
 \[ |\mathcal{L}_{A}(v)| \geq \frac{1}{2}\mathbb{E}[|\mathcal{L}_{A}(v)|]=\frac{\eta p^{2r}}{2}n=\frac{\eta \gamma^{2r}}{2^{2r+1}}n;\]
 \[d_A(v) \geq (1-\frac{\mu}{4})\left(c+\mu\right)np
    \geq \left(c+\frac{\mu}{2}\right)(1+\frac{\mu}{4})np\geq \left(c+\frac{\mu}{2}\right)|A|,\]
where the second inequality holds since $c+\mu\leq 1$.
Hence, $A$ is as desired by taking $\zeta=\frac{\eta\gamma^{2r}}{2^{2r+1}}$.
The lower bound of $|A|$ follows from $|\mathcal{L}_{A}(v)|\geq \zeta n$, thus $|A|\geq 2r\zeta n$.
\end{proof}

\subsubsection{Almost path-cover: Proof of Lemma~\ref{almost}}\label{sec5}
To prove Lemma \ref{almost}, we employ a recent result of Chen, Han, Wang, and Yang~\cite{CHWY2022}. For simplicity of presentation, we state a weaker version of their result as follows.
\begin{lemma}[\cite{CHWY2022}, Lemma 3.1]\label{lem4.1}
Given $\mu, \delta>0$, an $h$-vertex graph $H$ with $h\in \mathbb{N}$ and $h\geq 3$, there exists $\alpha>0$ such that the following holds for sufficiently large $n$.
Let $G$ be an $n$-vertex graph with $\delta(G)\geq \max\left\{\left(1-\frac{1}{ar(H)}+\mu\right)n, \left(\frac{1}{2}+\mu\right)n\right\}$ and $\alpha(G)\leq \alpha n$.
Then $G$ contains an $H$-tiling which covers all but at most $\delta n$ vertices.
\end{lemma}
%
In fact, Lemma \ref{almost} can be easily derived by applying Lemma \ref{lem4.1} with $H=P^r_{\ell}$, and making the trivial observation that $ar(P^r_{\ell})=g(r)$.
\footnote{Here, we remark that the proof of Lemma \ref{almost} also can be derived by combining a standard application of the powerful Szemer\'{e}di's regularity lemma and some involved embedding techniques as in \cite{CHWY2022}. In fact, the degree condition $(1-\tfrac{1}{r}+\mu )n$ enables us to apply the Hajnal--Szemer\'{e}di theorem and obtain an almost $K_{r}$-factor in the reduced graph. Then we suitably embed vertex-disjoint copies of $P^r_{\ell}$ into every collection of $r$ clusters that form a copy of $K_r$ as above, where one need some involved embedding techniques in the setting that $\alpha(G)=o(n)$ (e.g. the tree-building lemma in \cite{EHSS1983}).
}

\section{Connecting $r$-tuples}\label{sec6}
In this section we shall focus on Lemma~\ref{connect}, which allows us to ``connect'' two copies of $K_r$ (ordered) via an $r$-path of a constant length.
Following a standard application of the regularity lemma, we obtain an $\eps$-regular partition and then build a \emph{reduced graph} $R$ (see Definition~\ref{def2.7}).
We shall first use Lemma~\ref{lem6.2} to obtain a \emph{good walk} (See Definition \ref{def6.0}) in $R$ and then use some involved embedding techniques (see Lemma \ref{lem6.30} and Lemma \ref{lem6.31}) to finish the connection.
\subsection{Regularity}
In this paper, we make use of a degree form of the regularity lemma \cite{KS1996}.
We shall first introduce a basic notion of regular pair.
Given a graph $G$ and a pair $(V_{1}, V_{2})$ of vertex-disjoint subsets in $V(G)$, the \emph{density} of $(V_{1}, V_{2})$ is defined as

\begin{center}
$d(V_{1}, V_{2})=\frac{e(V_{1}, V_{2})}{|V_{1}||V_{2}|}$.\\
\end{center}

\begin{defn}[]\label{def2.1}
Given $\eps>0$, a graph $G$ and a pair $(V_{1}, V_{2})$ of vertex-disjoint subsets in $V(G)$, we say that the pair $(V_1, V_2)$ is $\eps $-\emph{regular} if for all $X\subseteq V_{1}$ and $Y\subseteq V_{2}$ satisfying
\[
|X| \ge \eps |V_{1}| ~\text{and}~  |Y| \ge \eps |V_{2}|,
\]
we have
\[
|d(X,Y) - d(V_1,V_2)|  \le  \eps.
\]
Moreover, a pair $(X, Y)$ is called $(\eps, \beta)$-\emph{regular} if $(X, Y)$ is $\eps$-regular and $d(X, Y)\geq \beta$.
\end{defn}

\begin{lemma}[\cite{KS1996}, Slicing Lemma]\label{lem2.5}
Assume $(V_{1}, V_{2})$ is $\eps$-regular with density $d$.
For some $\alpha\geq \eps$, let $V_{1}'\subseteq V_{1}$ with $|V_{1}'|\geq \alpha|V_{1}|$ and $V_{2}'\subseteq V_{2}$ with $|V_{2}'|\geq \alpha|V_{2}|$.
Then $(V_{1}', V_{2}')$ is $\eps'$-regular with $\eps':=\max\left\{2\eps, \eps/\alpha\right\}$ and for its density $d'$ we have $|d'-d|<\eps$.
\end{lemma}

\begin{lemma}[\cite{KS1996}, Degree form of Regularity Lemma]\label{lem2.6}
For every $\eps  > 0$, there is an $N = N(\eps )$ such that the following holds for any real number $\beta\in [0, 1]$ and $n\in \mathbb{N}$.
Let $G$ be a graph with $n$ vertices.
Then there exists a partition $V(G)=V_{0}\cup \cdots \cup V_{k} $ and a spanning subgraph $G' \subseteq G$ with the following properties:
\begin{enumerate}

         \item [$({\rm 1})$] $ \frac{1}{\eps}\leq k \le N $;

         \item [$({\rm 2})$] $|V_{i}| \le \eps  n$ for $i\in [0, k]$ and $|V_{1}|=|V_{2}|=\cdots=|V_{k}| =m$ for some $m\in \mathbb{N}$;

         \item [$({\rm 3})$] $d_{G'}(v) > d_{G}(v) - (\beta + \eps )n$ for all $v \in V(G)$;

         \item [$({\rm 4})$] each $V_{i}$ is an independent set in $G' $ for $i\in [k]$;

         \item [$({\rm 5})$] all pairs $(V_{i}, V_{j})$ are $\eps $-regular (in $G'$) with density 0 or at least $\beta$ for distinct $i, j\neq0$.

\end{enumerate}
\end{lemma}

More often, we call the partition in Lemma~\ref{lem2.6} an $(\eps,\beta)$-\emph{regular} partition. A widely-used auxiliary graph accompanied with the regular partition is the reduced graph.

\begin{defn}[Reduced graph]\label{def2.7}
Let $k\in \mathbb{N}$, $\beta, \eps>0$, $G$ be a graph with an $(\varepsilon, \beta)$-regular partition $V(G)=V_0\cup \ldots \cup V_k$ and $G'\subseteq G$ be a subgraph fulfilling the properties of Lemma \ref{lem2.6}.
We denote by $R_{\eps, \beta}$ the \emph{reduced graph} for the $(\eps,\beta)$-regular partition, which is defined as follows.
Let $V(R_{\eps, \beta})=\{V_{1}, \ldots, V_{k}\}$ and for two distinct clusters $V_{i}$ and $V_{j}$ we draw an edge between $V_{i}$ and $V_{j}$ if $d_{G'}(V_i, V_j)\geq \beta$ and no edge otherwise.
\end{defn}

The following fact presents a minimum degree of the reduced graph.
\begin{fac}[]\label{fact2.8}
Let $n\in \mathbb{N}$, $\eps, \beta, c\in(0,1)$ and $G$ be an $n$-vertex graph with $\delta(G)\ge cn$. Let $V(G)=V_{0}\cup \cdots \cup V_{k}$ be a vertex partition of $V(G)$ satisfying Lemma \ref{lem2.6} $(1)$-$(5)$ and $R:=R_{\eps, \beta}$ the reduced graph for the partition.
Then for every $V_{i}\in V(R)$ we have $d_{R}(V_{i})\ge(c-2\eps-\be)k$.
\end{fac}

\begin{proof}[Proof] Note that $|V_{0}|\leq \eps n$ and $|V_{i}|=m$ for every $i\in [k]$. Every edge in $R$ represents less than $m^{2}$ edges in $G'-V_{0}$. Thus we have
\begin{align}\label{al1}
    d_{R}(V_{i}) & \geq \frac{|V_{i}|(\delta(G)-(\beta+\eps)n-\eps n)}{m^{2}} \nonumber \\
    & \ge(c-2\eps-\be) k.\nonumber
\end{align}
\end{proof}

\subsection{Main tools}

\begin{defn}[Good walk]\label{def6.0}
Given $r, k, \ell\in \mathbb{N}$, let $R$ be a $k$-vertex graph.
Then we say $\mathcal{Q}=(S_{1}, \dots, S_{\ell})$ (allowing repetition) is a \emph{walk} of order $\ell$ in $R$, where $S_{i}\in V(R)$ for every $i\in [\ell]$, if every consecutive $r+1$ elements in $\mathcal{Q}$ induces a clique in $R$ of size either $r+1$ or $r$.
Moreover, in the latter case the repetition of a vertex is allowed only in adjacent positions.
Let $\mathcal{Q}[i, j]:=(S_{i}, S_{i+1}, \dots, S_{j})$ and $\mathcal{Q}[i]:=\mathcal{Q}[i,i]=S_i$ for every $i, j \in[\ell]$.
Then we call the $r$-tuples $\mathcal{Q}[1,r]$ and $\mathcal{Q}[\ell-r+1, \ell]$ the \emph{head} and \emph{tail} of $\mathcal{Q}$, respectively.
We say $S_{i}$ is a \emph{lazy element} in $\mathcal{Q}$ if $S_{i}=S_{i+1}$.
We say a walk $\mathcal{Q}$ is \emph{good} if the first lazy element is $S_{r+1}$, and for every two distinct lazy elements $S_{i}$ and $S_{j}$ (if any), we have $|i-j|\geq 20(r+1)$.
\end{defn}
Note that the parameter $20(r+1)$ in Definition~\ref{def6.0} is chosen to make the embedding process more convenient, and it could be an arbitrarily large constant.
Next, we give an example about Definition \ref{def6.0}.
Fix $r=2$ and $R$ to be a graph with $V(R)=\{V_{1}, V_{2}, \dots, , V_{6}, V_{7}\}$ such that $R[\{V_{i}, V_{i+1}, V_{i+2}\}]$ is a copy of $K_{3}$ for every $i\in\{1, 3, 4\}$ and $V_{6}V_{7}\in E(R)$.
Let $\mathcal{Q}=(V_{1}, V_{2}, V_{3}, V_{3}, V_{4}, V_{5}, V_{6}, V_{6}, V_{7})$.
Then by Definition \ref{def6.0} $\mathcal{Q}$ is a walk with lazy elements $\mathcal{Q}[3]$ and $\mathcal{Q}[7]$ but is not a good walk (See Figure \ref{L2}).
\begin{figure}[htbp]
\centering
\includegraphics[scale=0.6]{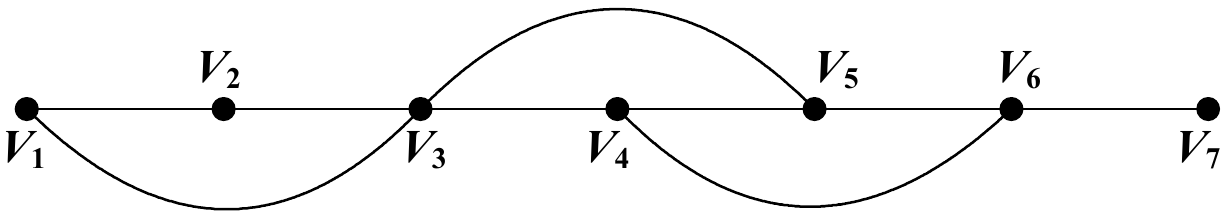}
\caption{$V_{1}V_{2}V_{3}V_{3}V_{4}V_{5}V_{6}V_{6}V_{7}$ is not a good walk}
\label{L2}
\end{figure}

The following lemma provides us a good walk with fixed head and tail.
\begin{lemma}\label{lem6.2}
For $\mu>0$ and $r\in\mathbb{N}$, the following holds for sufficiently large $k$.
Let $R$ be a $k$-vertex graph with $\delta(R)\geq \left(1-\frac{1}{r}+\mu\right)k$ and $\mf{x}$, $\mf{y}$ be two disjoint $r$-tuples of vertices in $R$, each inducing a copy of $K_{r}$.
Then there exists a good walk $\mathcal{Q}=(S_1,\ldots,S_{\ell})$ for some $\ell\leq 100r^{5}$ in $R$ with head $\mf{x}$ and tail $\mf{y}$.
Moreover the last lazy element of $\mathcal{Q}$, say $S_q$, satisfies that $q\leq \ell-20(r+1)$.
\end{lemma}

The proof of Lemma \ref{lem6.2} will be presented in next subsection.
The subsequent work in proving Lemma \ref{connect} is the embedding process.
To elaborate on this, we need the following two lemmas (Lemma \ref{lem6.30} and Lemma \ref{lem6.31}).
Figure \ref{L3} is an illustration of Lemma \ref{lem6.30}, and Figure \ref{L4} is an illustration of Lemma \ref{lem6.31}.
\begin{figure}[htbp]
\centering
\includegraphics[scale=0.6]{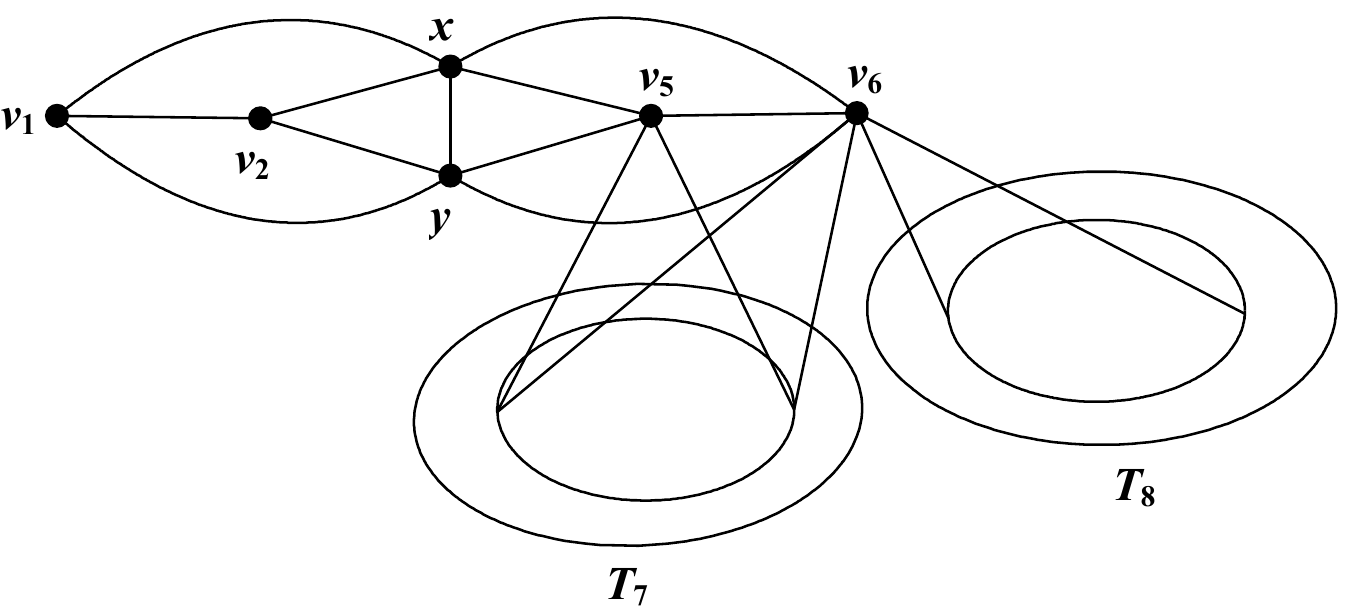}
\caption{Illustration of $r=2$}
\label{L3}
\end{figure}
\begin{lemma}\label{lem6.30}
For $\beta,\eta>0$ and $r, k\in\mathbb{N}$, there exist $\eps, \alpha>0$ such that the following holds for sufficiently large $m\in \mathbb{N}$.
Let $G$ be a graph with an equipartition $V(G)=V_{1}\cup V_{2}\cup\cdots\cup V_{k}$, $\alpha(G)\leq \alpha |V(G)|$, and $|V_{i}|=m$ for every $i\in [k]$.
Let $R$ be a graph with $V(R)=\{V_{1}, V_{2}, \dots, V_{k}\}$, and $V_{i}V_{j}\in E(R)$ if $(V_{i}, V_{j})$ is $(\eps, \beta)$-regular in $G$.
Let $\mathcal{Q}:=(S_{1}, S_{2}, \dots, S_{3r+2})$ be a good walk in $R$ with exactly one lazy element $S_{r+1}$ and $T_i\<S_i$ be a set of size at least $\eta m$ for every $i\in[3r+2]$ such that $T_{r+1}=T_{r+2}$.
Then there exists an $r$-path in $G$, say $P=v_{1}v_{2}\dots v_{r}xyv_{r+3}\dots v_{2r+2}$, such that $\{x, y\}\< T_{r+1}$, $v_{i}\in T_{i}$ for every $i\in[2r+2]\setminus \{r+1,r+2\}$, and $|N_{G}(v_{j-r}, \dots, v_{2r+2})\cap T_{j}|\geq \left(\frac{\beta}{2}\right)^{r}|T_j|$ for every $j\in [2r+3, 3r+2]$.
\end{lemma}

\begin{lemma}\label{lem6.31}
For $\beta,\eta>0$ and $r, k, \ell\in\mathbb{N}$, there exists $\eps>0$ such that the following holds for sufficiently large $m\in \mathbb{N}$.
Let $G$ be a graph with an equipartition $V(G)=V_{1}\cup V_{2}\cup \cdots\cup V_{k}$, and $|V_{i}|=m$ for every $i\in [k]$.
Let $R$ be a graph with $V(R)=\{V_{1}, V_{2}, \dots, V_{k}\}$, and $V_{i}V_{j}\in E(R)$ if $(V_{i}, V_{j})$ is $(\eps, \beta)$-regular in $G$.
Let $\mathcal{Q}:=(S_{1}, \dots, S_{\ell+r})$ be a good walk in $R$ without any lazy element and $T_i\<S_i$ be a set of size at least $\eta m$ for every $i\in[\ell+r]$.
Then there exists an $r$-path in $G$, say $P=v_{1}v_{2}\dots v_{\ell}$, such that $v_{i}\in T_{i}$ for every $i\in[\ell]$ and $|N_{G}(v_{j-r}, \dots, v_{\ell})\cap T_{j}|\geq \left(\frac{\beta}{2}\right)^{r}|T_j|$ for every $j\in [\ell+1, \ell+r]$.
\end{lemma}

\begin{figure}[htbp]
\centering
\includegraphics[scale=0.6]{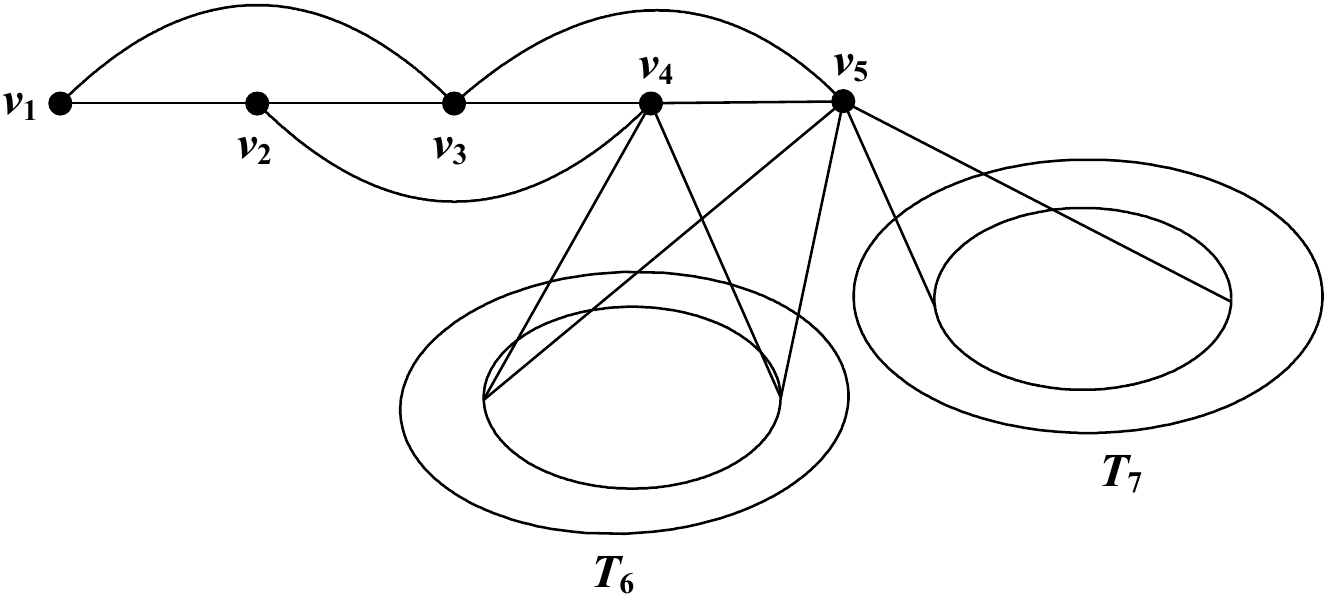}
\caption{Illustration of $r=2$ and $\ell=5$}
\label{L4}
\end{figure}

\subsection{Proof of Lemma~\ref{connect}}
We now have all the necessary tools to prove Lemma~\ref{connect}.
\begin{proof}[Proof of Lemma \ref{connect}]
Given $\mu>0$, $r\in \mathbb{N}$, we set $\beta=\frac{\mu}{20}$, and choose
\begin{center}
$\frac{1}{n}\ll \alpha\ll\frac{1}{k}\ll\eps\ll\mu, \frac{1}{r}$.
\end{center}
Let $G$ be an $n$-vertex graph with $\delta(G)\geq \left(1-\frac{1}{r}+\mu\right)n$ and $\alpha(G)\leq \alpha n$.
Let $\mf{x}$, $\mf{y}$ be two disjoint $r$-tuples of vertices, each inducing a copy of $K_{r}$ in $G$.
In order to simplify the notation, we write $\mf{x}=(x_{1}, \dots, x_{r})$ and $\mf{y}=(y_{r}, \dots, y_{1})$.
Our goal is to find an $r$-path $P$ in $G$ which connects $(x_{1}, \dots, x_{r})$ and $(y_{r}, \dots, y_{1})$.
Applying Lemma \ref{lem2.6} to $G$, we obtain an $(\eps,\be)$-regular partition $\mathcal{P}=\{V_{0}, V_{1}, \dots, V_{k}\}$ of $V(G)$ and write $m:=|V_{i}|$ for every $i\in [k]$.
Let $R:=R_{\eps, \beta}$ be the reduced graph for the $(\eps,\be)$-regular partition $\mathcal{P}$.
Then Fact~\ref{fact2.8} implies that $\delta(R)\geq \left(1-\frac{1}{r}+\frac{\mu}{2}\right)k$.

The following claim reduces the connecting process to finding a good walk in $R$ with fixed head and tail,
whose proof is postponed to the end of this subsection.
\begin{claim}\label{cl6.3}
Let $G[\{v_{1}, \dots, v_{r}\}]$ be a copy of $K_{r}$ in $G$ and $W\subseteq V(R)$ with $|W|\leq r$.
Then there exists a copy of $K_{r}$ in $R-W$, say $H$ with $V(H)=\{V_{1}, V_{2}, \dots, V_{r}\}$ such that $|N_{G}(v_{i}, \dots, v_{r})\cap V_{i}|\geq \mu m$ for every $i\in[r]$.
\end{claim}
Applying Claim \ref{cl6.3} with $W=\emptyset$, we obtain a copy of $K_r$ in $R$, whose vertex set, without loss of generality, is denoted as $\{V_{1}, \dots, V_{r}\}$.
Then
$|N_{G}(x_{i}, \dots, x_{r}) \cap V_{i}|\geq \mu m$ for every $i\in[r]$.
Applying Claim \ref{cl6.3} again with $W=\{V_{1}, \dots, V_{r}\}$, we obtain a copy of $K_r$ in $R-W$, whose vertex set is denoted as $\{U_{1},\dots, U_{r}\}\< V(R)\backslash W$.
Then $|N_{G}(y_{i}, \dots, y_{r})\cap U_{i}|\geq \mu m$ for every $i\in[r]$.

By applying Lemma \ref{lem6.2} to $R$, we obtain a good walk $\mathcal{Q}:=(S_1, \dots, S_\ell)$ for some $\ell\leq 100r^{5}$ with head $(V_{1}, \dots, V_{r})$
and tail $(U_{r},\dots, U_{1})$.
Let $S_{t_{1}}, S_{t_{2}}, \dots, S_{t_{s}}$ be all the lazy elements in $\mathcal{Q}$ for some $s\in \mathbb{N}$.
Then it follows that $t_{1}=r+1$, $t_{q+1}-t_{q}\geq 20(r+1)$ for every $q\in[s-1]$, $\ell-t_{s}\geq 20(r+1)$, and $s\leq\frac{100r^{5}}{20(r+1)}\leq  5r^{4}$.
The desired $r$-path will be constructed piece by piece in the following phases, where we denote by $P_0$ the $r$-path $x_1\ldots x_r$.

\noindent\textbf{Phase $0$.} Let $S^{0}_{i}:=N_{G}(x_{i}, \dots, x_{r})\cap V_{i}$ and $S^{0}_{\ell-i+1}:=N_{G}(y_{i}, \dots, y_{r})\cap U_{i}$ for every $i\in [r]$.
Then by Claim~\ref{cl6.3} we obtain that $|S^{0}_{i}|, |S^{0}_{\ell-r+i}|\geq \mu m$ for every $i\in [r]$.
Let $S^{0}_{i}:=S_{i}$ for every $i\in [r+1, \ell-r]$.

\noindent\textbf{Phase $1$.} We first consider the subwalk $\mathcal{Q}^{1}_{1}:=(S_{1}, \dots, S_{3r+2})$.
Applying Lemma \ref{lem6.30} to $\mathcal{Q}^{1}_{1}$ with $\eta=\mu$ and $T_i=S_i^0$ for every $i\in[3r+2]$, we obtain an $r$-path, say $P^{1}_{1}=v_{1}\cdots v_{r}x^{1}y^{1}v_{r+3}\cdots v_{2r+2}$ with $v_{i}\in S^0_{i}$ for every $i\in [2r+2]\backslash \{r+1, r+2\}$ and $x^{1}, y^{1}\in S^0_{r+1}=S^0_{r+2}$.
Moreover, we have
$|N_{G}(v_{j-r}, \dots, v_{2r+2})\cap S^0_{j}|\geq \left(\frac{\beta}{2}\right)^{r}m$ for every $j\in[2r+3, 3r+2]$.
Let $S^{1}_{j}=(N_{G}(v_{j-r}, \dots, v_{2r+2})\cap S^{0}_{j})\backslash V(P^{1}_{1})$ for every $j\in [2r+3, 3r+2]$.
Then $|S^{1}_{j}|\geq \left(\tfrac{\beta}{2}\right)^{r}m-|V(P_{1}^{1})|\geq\left(\frac{\beta}{4}\right)^{r}m$.

Next, we consider the subwalk $\mathcal{Q}^{2}_{1}:=(S_{2r+3}, \dots, S_{t_{2}-1})$.
Let $S^{1}_{j}:=S^{0}_{j}\backslash V(P_{1}^{1})$ for every $j\in [3r+3, t_{2}-1]$.
Then $|S^{1}_{j}|\geq m-|V(P_{1}^{1})|\geq\frac{m}{2}$ for every $j\in [3r+3, t_{2}-1]$.
Applying Lemma \ref{lem6.31} to $\mathcal{Q}^{2}_{1}$ with $\eta=\left(\frac{\be}{4}\right)^r$ and $T_i=S_{i+2r+2}^1$ for every $i\in[t_2-2r-3]$, we obtain an $r$-path, say $P^{2}_{1}=v_{2r+3}\cdots v_{t_{2}-r-1}$ with $v_{i}\in S^1_{i}$ for every $i\in [2r+3, t_{2}-r-1]$, such that \[|N_{G}(v_{j-r}, \dots, v_{t_{2}-r-1})\cap S^{1}_{j}|\geq \left(\tfrac{\beta}{2}\right)^{r}\tfrac{m}{2}~\text{for every}~ j\in [t_{2}-r, t_{2}-1].\]
By the definition of $S_j^0$ and $S_j^1$ for all $j\in[t_2-1]$, we end up with an $r$-path
\begin{center}
$x_1\ldots x_rv_{1}\cdots v_{r}x^{1}y^{1}v_{r+3}\cdots v_{t_2-r-1}$,
\end{center}
denoted as $P_{1}=P_0 P_{1}^{1} P_{1}^{2}$.
We update
\[S_{j}^{0}\leftarrow(N_{G}(v_{j-r}, \dots, v_{t_{2}-r-1})\cap S^{0}_{j})\backslash V(P_{1})~\text{for every}~j\in [t_{2}-r, t_{2}-1].\]
For every $j\in [t_{2}-r, t_{2}-1]$, we observe that
\begin{align}\label{ali1}
  |S_{j}^{0}| & =|(N_{G}(v_{j-r}, \dots, v_{t_{2}-r-1})\cap S^{0}_{j})\backslash V(P_{1})|\geq \left(\tfrac{\beta}{2}\right)^{r}\tfrac{m}{2}-|V(P_{1})|\geq \left(\tfrac{\beta}{4}\right)^{r}m.
\end{align}
For every $i\in [s-1]$, we keep updating $S_{j}^{0}$ for every $j\in [t_{i+1}-r, t_{i+1}-1]$, and for every $i\geq 1$, define ($\textbf{E}_{i}$) as follows and show that ($\textbf{E}_{i}$) holds in \textbf{Phase} $i$.
\begin{flushleft}
($\textbf{E}_{i}$). $|S^{0}_{j}|\geq \left(\frac{\beta}{4}\right)^{r}m$ for every $j\in [t_{i+1}-r, t_{i+1}-1]$.
\end{flushleft}
Therefore, ($\textbf{E}_{1}$) holds by (\ref{ali1}).

Suppose after the first $i-1$ ($i\geq 2$) phases, we obtain an $r$-path, say $P_{i-1}:=P_{0}P_{1}^{1} P_{1}^{2} \cdots P_{i-1}^{1} P_{i-1}^{2}$, and a sequence of subsets $S^{0}_{j}$ which satisfy ($\textbf{E}_{i-1}$).
Next we move to \textbf{Phase} $i$.

\noindent\textbf{Phase $i$ ($i\geq 2$).}
We write $S^{0}_{j}:=S_{j}\backslash V(P_{i-1})$ for every $j\in [t_{i}, t_{i+1}-1]$, and thus $|S^{0}_{j}|\geq m-|V(P_{i-1})|\geq\frac{m}{2}$.
Similarly, we focus on the subwalks \[\mathcal{Q}^{1}_{i}:=(S_{t_{i}-r}, \dots, S_{t_{i}+2r+1})~\text{and}~ \mathcal{Q}^{2}_{i}:=(S_{t_{i}+r+2}, \dots, S_{t_{i+1}-1}).\]
Recall $S_{t_{i}}$ is a lazy element.
By applying Lemma~\ref{lem6.30} and Lemma~\ref{lem6.31} as in \textbf{Phase} $1$ and the condition in ($\textbf{E}_{i-1}$), we end up with an $r$-path, denoted as $P_i$ such that \[P_i=P_{i-1} v_{t_{i}-r}\cdots v_{t_{i}-1}x^{i}y^{i}v_{t_{i}+2}\cdots v_{t_{i}+r+1} v_{t_{i}+r+2}\dots v_{t_{i+1}-r-1},\] where $x^{i}, y^{i}\in S^0_{t_{i}}$ and $v_{j}\in S^0_{j}$ for every $j\in[t_{i}-r, t_{i+1}-r-1]\backslash \{t_{i}, t_{i}+1\}$.
Moreover, we obtain a sequence of subsets $N_{G}(v_{j-r}, \dots, v_{t_{i+1}-r-1})\cap S^{0}_{j}$, $j\in [t_{i+1}-r, t_{i+1}-1]$, each of size at least $\left(\tfrac{\beta}{2}\right)^{r}\tfrac{m}{2}$.
Similarly, we update \[S_{j}^{0}\leftarrow(N_{G}(v_{j-r}, \dots, v_{t_{i+1}-r-1})\cap S^{0}_{j})\backslash V(P_{i})~\text{for every}~j\in [t_{i+1}-r, t_{i+1}-1].\]
It holds that $|S_{j}^{0}|=|(N_{G}(v_{j-r}, \dots, v_{t_{i+1}-r-1})\cap S^{0}_{j})\backslash V(P_{i})|\geq \left(\tfrac{\beta}{2}\right)^{r}\tfrac{m}{2}-|V(P_{i})|\geq \left(\tfrac{\beta}{4}\right)^{r}m$ for every $j\in [t_{i+1}-r, t_{i+1}-1]$.
Now ($\textbf{E}_i$) holds and we go on to \textbf{Phase} $i+1$.

Suppose after \textbf{Phase} $s$ we end up with an $r$-path, denoted as \[P_{s}=P_{s-1} v_{t_{s}-r}\cdots v_{t_{s}-1}x^{s}y^{s}v_{t_{s}+2}\cdots v_{t_{s}+r+1}  v_{t_{s}+r+2}\dots v_{\ell-r},\] and a sequence of $r$ subsets $N_{G}(v_{j-r}, \dots, v_{\ell-r})\cap S^{0}_{j}$, $j\in [\ell-r+1, \ell]$, each of size at least $\left(\tfrac{\beta}{2}\right)^{r}|S_j^0|\ge \left(\tfrac{\beta}{2}\right)^{r}\mu m$.
Again, we update
\begin{center}
$S_{j}^{0}$ $\leftarrow$ $(N_G(v_{j-r}, \dots, v_{\ell-r})\cap S^{0}_{j})\backslash V(P_{s})$ for every $j\in [\ell-r+1, \ell]$.
\end{center}
Then each $S^{0}_{j}$ has cardinality $|S^{0}_{j}|\ge \left(\tfrac{\beta}{2}\right)^{r}\mu m-|V(P_s)|\geq \left(\tfrac{\beta}{4}\right)^{r}\mu m$.
Recall that $S^{0}_{\ell-i+1}\<U_i$ for every $i\in[r]$ and all pairs $(U_i,U_j)$ are $\eps$-regular with density at least $\be$ for distinct $i,j\in[r]$.
Then as $\eps\ll \be,\mu,\frac{1}{r}$, Lemma \ref{lem2.5} implies that all pairs  $(S^{0}_{i}, S^{0}_{j})$ are $\eps'$-regular with density at least $\be-\eps$ for some $\eps':=\max\{2\eps, \tfrac{|S_{i}|}{|S^{0}_{i}|}\eps\}\le\left(\frac{4}{\beta}\right)^r\frac{\eps}{\mu}$ for distinct $i, j\in [\ell-r+1, \ell]$.
Then we can greedily find a copy of $K_{r}$, say $H_{3}$ with $V(H_{3})=\{v_{\ell-r+1}, \dots, v_{\ell}\}$, such that $v_{j}\in S^{0}_{j}$ for every $j\in [\ell-r+1, \ell]$.
Since $S^{0}_{\ell-i+1}=N_{G}(y_{i}, \dots, y_{r})\cap U_i$ for every $i\in [r]$ (see \textbf{Phase} $0$), we obtain an $r$-path $P_{s} v_{\ell-r+1}\cdots v_{\ell} y_{r}\cdots y_{1}$ of length at most $2\ell\leq 200r^{5}$ as desired.
This completes the proof of Lemma~\ref{connect}.
\end{proof}

Now it remains to prove Claim \ref{cl6.3}.

\begin{proof}[Proof of Claim \ref{cl6.3}]
Since $\de(G)\ge (1-\frac{1}{r}+\mu)n$, $v_{1}, \dots,v_{r}$ have at least $r\mu n$ common neighbors in $G$.
Hence, there exists a cluster, say $V_{1}\in V(R)\backslash W$, such that $|N_G(v_{1}, \dots,v_{r})\cap V_{1}|\geq \frac{r\mu n-|W|m}{k-|W|}\geq\frac{r\mu n-rm}{k}\geq\mu m$, as $\frac{1}{k}\ll \mu$.
Suppose we have obtained a maximal collection of clusters in $V(R)\backslash W$, say $V_{1}, \dots, V_s$ for some $s\ge 1$ such that they form a copy of $K_s$ in $R$ and $|N_G(v_{i}, \dots, v_{r})\cap V_{i}|\geq \mu m$ for every $i\in[s]$. Suppose for a contradiction that $s<r$.
For every $V'\in N_R(V_{1}, \dots, V_{s})\backslash W$, we have $|(N_G(v_{s+1}, \dots, v_{r})\backslash\bigcup_{V_{i}\in W}V_{i})\cap V'|<\mu m$.
Since $\frac{1}{k}\ll \mu, \frac{1}{r}$, it follows that
\begin{center}
$|N_R(V_{1}, \dots, V_{s})\backslash W|\geq s\left(1-\frac{1}{r}+\frac{\mu}{2}\right)k-(s-1)k-r\geq\left(\frac{r-s}{r}+\frac{s}{4}\mu\right)k>\frac{r-s}{r} k$.
\end{center}
Then we have
\begin{align*}
       |(N_G(v_{s+1}, \dots, v_{r})\backslash \bigcup_{V_{i}\in W}V_{i})| & <\mu m|N_R(V_{1}, \dots, V_{s})\backslash W|+|V(G)\backslash(\bigcup_{V_{i}\in W\cup N_R(V_{1}, \dots, V_{s})}V_{i})|\\
       &= \mu m|N_R(V_{1}, \dots, V_{s})\backslash W|+n-|W|m-|N_R(V_{1}, \dots, V_{s})\backslash W|m\\
       &<(\mu-1)m\frac{(r-s)k}{r}+n-|W|m.
     \end{align*}

\begin{flushleft}
Meanwhile, we have
\end{flushleft}
\begin{align*}
  |N_G(v_{s+1}, \dots, v_{r})\backslash \bigcup_{V_{i}\in W}V_{i})|\geq & (r-s)(1-\tfrac{1}{r}+\mu)n-(r-s-1)n-|W|m\\=&n-|W|m-(r-s)(\frac{1}{r}-\mu)n.
\end{align*}
Putting the bounds together, we obtain
\begin{center}
$(r-s)(\frac{1}{r}-\mu)n>(1-\mu)m\frac{(r-s)k}{r}$.
\end{center}
Since $mk\geq (1-\varepsilon)n$, we get $1-r\mu>(1-\mu)(1-\varepsilon)>1-2\mu$, a contradiction.
%
%
%
\end{proof}

%
%

\subsection{Proof of Lemma \ref{lem6.2}}
The proof of Lemma \ref{lem6.2} goes roughly as follows.
Let $\mf{x},\mf{y}$ be two disjoint $r$-tuples of vertices in $R$, each inducing a copy of $K_r$, say $H_{\mf{x}}$, $H_{\mf{y}}$, respectively.
We first build a sequence of copies of $K_{r}$ starting from $H_{\mf{x}}$ and ending with $H_{\mf{y}}$, where any two consecutive copies share $r-1$ common vertices.
Then we extend every copy of $K_{r}$ in the sequence to $K_{r+1}$ such that every two consecutive copies of $K_{r+1}$ share exactly $r-1$ common vertices (see Lemma~\ref{lem6.3}).
Consider the resulting sequence of copies of $K_{r+1}$, say $H_1, H_2, \ldots, H_{s}$, with $|V(H_{i})\cap V(H_{i+1})|=r-1$ for every $i\in[s-1]$, and then define two permutations $\mf{x}^+,\mf{y}^+$ of $V(H_1), V(H_s)$, respectively, satisfying $\mf{x}^+[1,r]=\mf{x}$ and $\mf{y}^+[2,r+1]=\mf{y}$.
The good walk comes from the following two phases:
\begin{enumerate}
  \item [$(1)$] For every $i\in [s-1]$, we build a good walk $\mathcal{Q}_{i}$ of order $\ell_i$ (see Lemma~\ref{lem6.6}), such that
      \begin{itemize}
        \item $\mathcal{Q}_{1}[1,r+1]=\mf{x}^+$ and $\mathcal{Q}_{i}[1,r+1]=\mathcal{Q}_{i-1}[\ell_{i-1}-r, \ell_{i-1}]$ when $i\ge 2$;
        \item $\mathcal{Q}_{i}[\ell_{i}-r, \ell_{i}]$ is a permutation of $V(H_{i+1})$, and $V(\mathcal{Q}_{i})\subseteq V(H_{i})\cup V(H_{i+1})$.
      \end{itemize}
  \item [$(2)$] We build a good walk $\mathcal{Q}_{s}$ which starts with $\mathcal{Q}_{s-1}[\ell_{s-1}-r, \ell_{s}]$, i.e. a permutation of $V(H_s)$ obtained in $(1)$, and ends with the fixed permutation $\mf{y}^+$ of $V(H_s)$ (see Lemma~\ref{lem6.5}).
\end{enumerate}
Hence we obtain a desired good walk with head $\mf{x}$ and tail $\mf{y}$ by piecing together all these walks $\mathcal{Q}_{1},\ldots,\mathcal{Q}_{s}$.
Now we give the three lemmas.

\begin{lemma}\label{lem6.3}
Given $\mu>0$, $r\in\mathbb{N}$, the following holds for sufficiently large $k$.
Let $R$ be a $k$-vertex graph with $\delta(R)\geq \left(1-\frac{1}{r}+\mu\right)k$, $H_{1}$ and $H_{2}$ be two vertex-disjoint copies of $K_{r}$ in $R$.
Then there exists a family of copies of $K_{r+1}$, say $\{B_{1}, B_{2},\dots, B_{\ell}\}$ for some $\ell\leq r^{2}$, with
\begin{itemize}
  \item $V(H_{1})\< V(B_{1})$, $V(H_{2})\< V(B_{\ell})$;
  \item $|V(B_{i})\cap V(B_{i+1})|=r-1$ for every $i\in [\ell-1]$.
\end{itemize}
\end{lemma}

\begin{proof}[Proof]
Given $\mu >0$, $r\in \mathbb{N}$ and $r\geq 2$, we choose $\frac{1}{k}\ll \mu, \frac{1}{r}$.
Let $R$ be a $k$-vertex graph with $\delta(R)\geq \left(1-\frac{1}{r}+\mu\right)k$.
Let $V(H_{1})=\{v_{1}, \dots, v_{r}\}$ and $V(H_{2})=\{u_{1}, \dots, u_{r}\}$.
It suffices to find a family of copies of $K_{r}$, say $\{A_{1}, \dots, A_{\ell}\}$ for some $\ell\leq r^{2}$ such that
\begin{itemize}
  \item $H_{1}=A_{1}$ and $H_{2}=A_{\ell}$;
  \item $|V(A_{i})\cap V(A_{i+1})|=r-1$ for every $i\in [\ell-1]$.
\end{itemize}
In fact, the desired copies $B_i$ can be easily obtained by extending $A_i$ to a copy of $K_{r+1}$ by a new vertex, which is possible because every $r$ vertices in $R$ have at least $r\mu k$ common neighbors.
To achieve this, we first prove the following claim.
We set $t_{0}:=0$, $t_{1}:=r-1$, $t_{i+1}:=t_{i}+(r-i-1)$ for every $i\in [r-2]$ and let $Q_{0}:=(v_{1}, \dots, v_{r})$.

\begin{figure}
\begin{center}
\tikzset{every picture/.style={line width=0.75pt}} 

\begin{tikzpicture}[x=0.75pt,y=0.75pt,yscale=-1,xscale=1]

\draw   (60,100) .. controls (60,86.19) and (89.1,75) .. (125,75) .. controls (160.9,75) and (190,86.19) .. (190,100) .. controls (190,113.81) and (160.9,125) .. (125,125) .. controls (89.1,125) and (60,113.81) .. (60,100) -- cycle ;
\draw   (320,100) .. controls (320,77.91) and (371.49,60) .. (435,60) .. controls (498.51,60) and (550,77.91) .. (550,100) .. controls (550,122.09) and (498.51,140) .. (435,140) .. controls (371.49,140) and (320,122.09) .. (320,100) -- cycle ;
\draw  [fill={rgb, 255:red, 66; green, 102; blue, 102 }  ,fill opacity=0.1 ] (160,190) .. controls (160,178.95) and (182.39,170) .. (210,170) .. controls (237.61,170) and (260,178.95) .. (260,190) .. controls (260,201.05) and (237.61,210) .. (210,210) .. controls (182.39,210) and (160,201.05) .. (160,190) -- cycle ;
\draw  [fill={rgb, 255:red, 189; green, 126; blue, 74 }  ,fill opacity=0.1 ] (310,180) .. controls (310,171.72) and (325.67,165) .. (345,165) .. controls (364.33,165) and (380,171.72) .. (380,180) .. controls (380,188.28) and (364.33,195) .. (345,195) .. controls (325.67,195) and (310,188.28) .. (310,180) -- cycle ;
\draw  [fill={rgb, 255:red, 16; green, 125; blue, 172 }  ,fill opacity=0.1 ] (440,160) .. controls (440,154.48) and (448.95,150) .. (460,150) .. controls (471.05,150) and (480,154.48) .. (480,160) .. controls (480,165.52) and (471.05,170) .. (460,170) .. controls (448.95,170) and (440,165.52) .. (440,160) -- cycle ;
\draw    (80,100) ;
\draw [shift={(80,100)}, rotate = 0] [color={rgb, 255:red, 0; green, 0; blue, 0 }  ][fill={rgb, 255:red, 0; green, 0; blue, 0 }  ][line width=0.75]      (0, 0) circle [x radius= 1.34, y radius= 1.34]   ;
\draw    (170,100) ;
\draw [shift={(170,100)}, rotate = 0] [color={rgb, 255:red, 0; green, 0; blue, 0 }  ][fill={rgb, 255:red, 0; green, 0; blue, 0 }  ][line width=0.75]      (0, 0) circle [x radius= 1.34, y radius= 1.34]   ;
\draw    (110,100) ;
\draw [shift={(110,100)}, rotate = 0] [color={rgb, 255:red, 0; green, 0; blue, 0 }  ][fill={rgb, 255:red, 0; green, 0; blue, 0 }  ][line width=0.75]      (0, 0) circle [x radius= 1.34, y radius= 1.34]   ;
\draw    (140,100) ;
\draw [shift={(140,100)}, rotate = 0] [color={rgb, 255:red, 0; green, 0; blue, 0 }  ][fill={rgb, 255:red, 0; green, 0; blue, 0 }  ][line width=0.75]      (0, 0) circle [x radius= 1.34, y radius= 1.34]   ;
\draw    (360,100) ;
\draw [shift={(360,100)}, rotate = 0] [color={rgb, 255:red, 0; green, 0; blue, 0 }  ][fill={rgb, 255:red, 0; green, 0; blue, 0 }  ][line width=0.75]      (0, 0) circle [x radius= 1.34, y radius= 1.34]   ;
\draw    (410,100) ;
\draw [shift={(410,100)}, rotate = 0] [color={rgb, 255:red, 0; green, 0; blue, 0 }  ][fill={rgb, 255:red, 0; green, 0; blue, 0 }  ][line width=0.75]      (0, 0) circle [x radius= 1.34, y radius= 1.34]   ;
\draw    (460,100) ;
\draw [shift={(460,100)}, rotate = 0] [color={rgb, 255:red, 0; green, 0; blue, 0 }  ][fill={rgb, 255:red, 0; green, 0; blue, 0 }  ][line width=0.75]      (0, 0) circle [x radius= 1.34, y radius= 1.34]   ;
\draw    (510,100) ;
\draw [shift={(510,100)}, rotate = 0] [color={rgb, 255:red, 0; green, 0; blue, 0 }  ][fill={rgb, 255:red, 0; green, 0; blue, 0 }  ][line width=0.75]      (0, 0) circle [x radius= 1.34, y radius= 1.34]   ;
\draw    (180,190) ;
\draw [shift={(180,190)}, rotate = 0] [color={rgb, 255:red, 0; green, 0; blue, 0 }  ][fill={rgb, 255:red, 0; green, 0; blue, 0 }  ][line width=0.75]      (0, 0) circle [x radius= 1.34, y radius= 1.34]   ;
\draw    (210,190) ;
\draw [shift={(210,190)}, rotate = 0] [color={rgb, 255:red, 0; green, 0; blue, 0 }  ][fill={rgb, 255:red, 0; green, 0; blue, 0 }  ][line width=0.75]      (0, 0) circle [x radius= 1.34, y radius= 1.34]   ;
\draw    (240,190) ;
\draw [shift={(240,190)}, rotate = 0] [color={rgb, 255:red, 0; green, 0; blue, 0 }  ][fill={rgb, 255:red, 0; green, 0; blue, 0 }  ][line width=0.75]      (0, 0) circle [x radius= 1.34, y radius= 1.34]   ;
\draw [color={rgb, 255:red, 195; green, 39; blue, 43 }  ,draw opacity=1 ]   (110,100) -- (180,190) ;
\draw [color={rgb, 255:red, 195; green, 39; blue, 43 }  ,draw opacity=1 ]   (140,100) -- (180,190) ;
\draw [color={rgb, 255:red, 195; green, 39; blue, 43 }  ,draw opacity=1 ]   (170,100) -- (180,190) ;
\draw  [color={rgb, 255:red, 66; green, 102; blue, 102 }  ,draw opacity=1 ] (155.37,213.18) .. controls (143.49,192.69) and (183.31,147.42) .. (244.3,112.08) .. controls (305.3,76.74) and (364.37,64.7) .. (376.24,85.2) .. controls (388.12,105.69) and (348.3,150.96) .. (287.31,186.3) .. controls (226.32,221.64) and (167.24,233.68) .. (155.37,213.18) -- cycle ;
\draw  [color={rgb, 255:red, 189; green, 126; blue, 74 }  ,draw opacity=1 ] (311.03,132.97) .. controls (334.83,98.39) and (378.38,78.14) .. (408.3,87.75) .. controls (438.23,97.35) and (443.2,133.17) .. (419.4,167.75) .. controls (395.61,202.33) and (352.06,222.57) .. (322.14,212.97) .. controls (292.21,203.37) and (287.24,167.55) .. (311.03,132.97) -- cycle ;
\draw  [color={rgb, 255:red, 16; green, 125; blue, 172 }  ,draw opacity=1 ] (358.96,129.22) .. controls (338.59,97.06) and (352.32,69.47) .. (389.64,67.6) .. controls (426.96,65.73) and (473.73,90.29) .. (494.1,122.44) .. controls (514.48,154.6) and (500.74,182.18) .. (463.43,184.06) .. controls (426.11,185.93) and (379.34,161.37) .. (358.96,129.22) -- cycle ;
\draw    (330,180) ;
\draw [shift={(330,180)}, rotate = 0] [color={rgb, 255:red, 0; green, 0; blue, 0 }  ][fill={rgb, 255:red, 0; green, 0; blue, 0 }  ][line width=0.75]      (0, 0) circle [x radius= 1.34, y radius= 1.34]   ;
\draw    (360,180) ;
\draw [shift={(360,180)}, rotate = 0] [color={rgb, 255:red, 0; green, 0; blue, 0 }  ][fill={rgb, 255:red, 0; green, 0; blue, 0 }  ][line width=0.75]      (0, 0) circle [x radius= 1.34, y radius= 1.34]   ;
\draw    (460,160) ;
\draw [shift={(460,160)}, rotate = 0] [color={rgb, 255:red, 0; green, 0; blue, 0 }  ][fill={rgb, 255:red, 0; green, 0; blue, 0 }  ][line width=0.75]      (0, 0) circle [x radius= 1.34, y radius= 1.34]   ;
\draw [color={rgb, 255:red, 195; green, 39; blue, 43 }  ,draw opacity=1 ]   (210,190) .. controls (226.44,230.33) and (311.78,229.67) .. (330,180) ;
\draw [color={rgb, 255:red, 195; green, 39; blue, 43 }  ,draw opacity=1 ]   (240,190) .. controls (261.78,203.67) and (290,210) .. (330,180) ;
\draw [color={rgb, 255:red, 195; green, 39; blue, 43 }  ,draw opacity=1 ]   (360,180) -- (460,160) ;

\draw (70,80.4) node [anchor=north west][inner sep=0.75pt]  [font=\footnotesize]  {$v_{1}$};
\draw (100,80.4) node [anchor=north west][inner sep=0.75pt]  [font=\footnotesize]  {$v_{2}$};
\draw (130.67,80.4) node [anchor=north west][inner sep=0.75pt]  [font=\footnotesize]  {$v_{3}$};
\draw (160,80.4) node [anchor=north west][inner sep=0.75pt]  [font=\footnotesize]  {$v_{4}$};
\draw (360,80.4) node [anchor=north west][inner sep=0.75pt]  [font=\footnotesize]  {$u_{1}$};
\draw (410,80.4) node [anchor=north west][inner sep=0.75pt]  [font=\footnotesize]  {$u_{2}$};
\draw (460,80.4) node [anchor=north west][inner sep=0.75pt]  [font=\footnotesize]  {$u_{3}$};
\draw (510,80.4) node [anchor=north west][inner sep=0.75pt]  [font=\footnotesize]  {$u_{4}$};
\draw (225.33,152.07) node [anchor=north west][inner sep=0.75pt]    {$Q_{1}$};
\draw (350.67,147.73) node [anchor=north west][inner sep=0.75pt]    {$Q_{2}$};
\draw (477.33,139.07) node [anchor=north west][inner sep=0.75pt]    {$Q_{3}$};
\end{tikzpicture}
\end{center}
\caption{Illustration of Claim \ref{cl5.12} with $r=4$}
\end{figure}
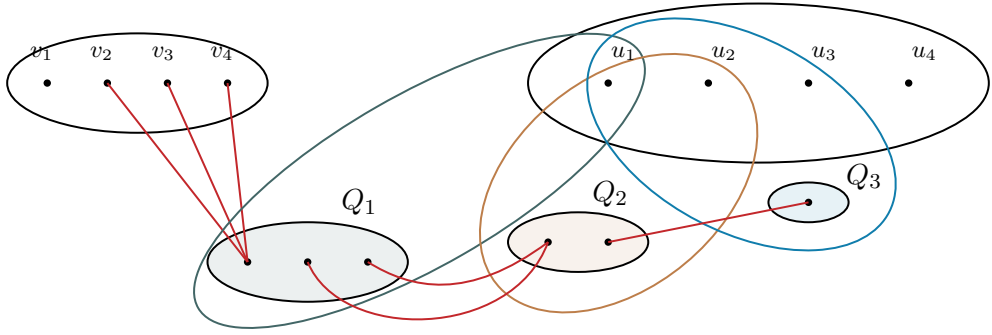

\begin{claim}\label{cl5.12}
There exist~$r-1$~sequences of vertices in $R$ (mutually vertex disjoint), say $Q_{1}, \dots, Q_{r-1}$ with $Q_{i}:=(x_{t_{i-1}+1}, \dots, x_{t_{i}})$, $|Q_{i}|=r-i$, satisfying that $V(Q_{i})\< N_R(u_{1}, \dots, u_{i})$ and $Q_{i-1}Q_{i}$ induces an $(r-i)$-power of a path in $R$ for every $i\in [r-1]$.
\end{claim}

\begin{pr}
We shall iteratively build each $Q_i$ for $i\in[r-1]$ as required.
For the case $i=1$, since $|N_{R}(v_{2}, \dots, v_{r}, u_{1})|\geq r\mu k$,
we arbitrarily take a common neighbor of $v_{2}, \dots, v_{r}$ and $u_{1}$, say $x_{1}$.
Similarly we can iteratively pick another $r-2$ distinct vertices say $x_{2},\ldots,x_{r-2}$, such that each $x_j$ is a common neighbor of $v_{j+1}, \dots, v_{r}, x_{1}, \dots, x_{j-1}$ and $u_{1}$ whilst avoiding all the vertices in $\{u_1,\ldots, u_r\}$.
Then $Q_{1}:=(x_{1}, \dots, x_{r-1})$ is desired as $Q_{0}Q_{1}=v_1v_2\ldots v_r x_1\ldots x_{r-1}$ actually induces an $(r-1)$-power of a path in $R$.

Suppose we have obtained $Q_1,\ldots,Q_q$ as above with $Q_{i}:=(x_{t_{i-1}+1}, \dots, x_{t_{i}})$, $V(Q_{i})\< N_R(u_{1}, \dots, u_{i})$, and $Q_{i-1}Q_{i}$ induces an $(r-i)$-power of a path in $R$ for every $i\in[q]$.
Now we will find the existence of $Q_{q+1}$.
Similarly, since $|\bigcup\limits^q_{i=1}U_i|\leq q(r-1)<r\mu k$, we can iteratively pick $r-i-1$ distinct vertices say $x_{t_{q}+1}, \dots, x_{t_{q+1}}$
such that each $x_{t_{q}+j}$ ($j\in[r-i-1]$) is a common neighbor of the $r$ vertices $x_{t_{q-1}+j+1}, \dots, x_{t_{q}}, \dots, x_{t_{q}+j-1}, u_{1}, \dots, u_{q+1}$.
This yields the desired $Q_{q+1}=(x_{t_{q}+1}, \dots, x_{t_{q+1}})$.
\end{pr}

Next, we find all desired copies of $K_{r}$.
By Claim \ref{cl5.12}, for every $i\in [r-1]$, every consecutive $r-i+1$ vertices in the sequence $Q_{i-1}Q_{i}$ together with $\{u_{1}, \dots, u_{i-1}\}$ form a copy of $K_{r}$, which yields a sequence of copies of $K_{r}$ in order, say $\mathcal{L}^{i}:=(A^{i}_{1}, \dots, A^{i}_{r-i+1})$.
Observe that every two consecutive copies of $K_{r}$ in $\mathcal{L}^{i}$ share $r-1$ vertices consisting of $u_{1}, \dots, u_{i-1}$ and $r-i$ consecutive vertices in the sequence $Q_{i-1}Q_{i}$.
Also, since $V(Q_{i})\< N_R(u_{1}, \dots, u_{i})$, we have that $V(Q_{r-1})\cup\{u_{1}, \dots, u_{r-1}\}$ induces a copy of $K_r$, denoted as $A^r_1$.
We claim that the resulting sequence 
\[A^{1}_{1} \dots A^{1}_{r} \dots A^{r-1}_{1} A^{r-1}_{2} A^r_1 H_{2}\] is desired.
In fact, since every two consecutive copies in each $\mathcal{L}_{i}$ share $r-1$ vertices, it remains to show that the last element of $\mathcal{L}_{i}$ and the first element of $\mathcal{L}_{i+1}$ share $r-1$ vertices. This easily follows as $V(A^{i}_{r-i+1})=\{x_{t_{i-1}}, x_{t_{i-1}+1}, \dots, x_{t_{i}}, u_{1}, \dots, u_{i-1}\}$ and $V(A^{i+1}_{1})=\{x_{t_{i-1}+1}, \dots, x_{t_{i}}, u_{1}, \dots, u_{i}\}$ for every $i\in [r-1]$.
Observe that $|V(A^r_1)\cap V(H_{2})|=r-1$.

The proof is completed by renaming the copies of $K_{r}$ in the above sequence as $A_{1}, \dots, A_{\ell}$ in order, where $\ell:=2+\sum^{r-1}_{i=1}|\mathcal{L}^{i}|=\frac{r^{2}+r+2}{2}\leq r^{2}$.
\end{proof}

\begin{lemma}\label{lem6.5}
Let $H$ be a copy of $K_{r+1}$.
Then for any $\pi_{1}, \pi_{2}\in \mathbf{S}_{V(H)}$, there exists a good walk $\mathcal{Q}$ in $H$ with order $\ell\leq 65r^{3}$ such that
\begin{itemize}
  \item $\mathcal{Q}[1,r+1]=\pi_{1}$ and $\mathcal{Q}[\ell-r, \ell]=\pi_{2}$;
  \item There is no lazy element among the last $20r+20$ elements of the walk.
\end{itemize}
\end{lemma}
\begin{proof}[Proof]
Let $V(H)=\{v_{1}, \dots, v_{r+1}\}$.
For any $\pi_{1}, \pi_{2}\in \mathbf{S}_{V(H)}$, we can transform $\pi_{1}$ to $\pi_{2}$ by recursively swapping two consecutive elements at most $r^{2}$ times.
Let $q_{1}, \dots, q_{s}$ be the sequence of all the positions where we make swapping as above with $q_{i}\in [r]$ for every $i\in[s]$.
To prove Lemma \ref{lem6.5}, it suffices to prove that for every $q\in[r]$ and every $\pi_{3}\in \mathbf{S}_{V(H)}$, there exists a good walk $\mathcal{Q}_q$ with two ends $\pi_{3}$ and $\pi_{4}$ where $\pi_{4}=\pi_{3}\circ(q, q+1)$.
Without loss of generality, we assume $\pi_{3}=(v_{1}, \dots, v_{q-1}, v_{q}, v_{q+1}, \dots, v_{r+1})$ and $\pi_{4}=(v_{1}, \dots, v_{q-1}, v_{q+1}, v_{q}, \dots, v_{r+1})$.
We build the walk $\mathcal{Q}_q=(S_{1}, S_{2}, \dots)$ as follows.

\begin{itemize}
  \item If $q\in[r-1]$, then we build $\mathcal{Q}_q=\pi_3 (v_{r+1})\pi_4\pi_4\cdots $
      where we repeat $\pi_4$ 21 times to guarantee that there is no lazy element among the last $20r+20$ elements.
  Observe that $\mathcal{Q}_q$ is a good walk with two ends $\pi_{3}$ and $\pi_{4}$, exactly one lazy element $S_{r+1}$, and $|\mathcal{Q}_q|=22r+23$.
  \item If $q=r$, then let $\pi_5:=(v_{r+1}, v_2, \cdots, v_r, v_1)$, and we build $\mathcal{Q}_q=\pi_3\pi_5\cdots\pi_5\pi_4\cdots\pi_4$ where we repeat $\pi_5$ 21 times to guarantee that the distance between two lazy elements is bigger than $20(r+1)$, and repeat $\pi_4$ 21 times to guarantee that there is no lazy element among the last $20r+20$ elements.
     Observe that $\mathcal{Q}_q$ is a good walk with two ends $\pi_{3}$ and $\pi_{4}$, two lazy elements $S_{r+1}$ and $S_{22r+22}$, and $|\mathcal{Q}_r|=43r+43$.
\end{itemize}
Hence, for any $\pi_{1}, \pi_{2}\in \mathbf{S}_{V(H)}$, we can transform $\pi_{1}$ to $\pi_{2}$ by performing a sequence of at most $r^{2}$ switchings, each of which switches two consecutive elements as above.
The concatenation of the corresponding good walks gives rise to a good walk $\mathcal{Q}$ with two ends $\pi_{1}$ and $\pi_{2}$, where we have $|\mathcal{Q}|\leq (43r+43)r^{2}\leq 65r^{3}$.
\end{proof}

\begin{lemma}\label{lem6.6}
Let $H_{1}, H_{2}$ be two copies of $K_{r+1}$ in $R$ with $|V(H_{1})\cap V(H_{2})|=r-1$.
Then for any given permutation $\pi$ of $V(H_{1})$ there exists a good walk $\mathcal{Q}$ of order $\ell\le 100r^3$ such that $V(\mathcal{Q})=V(H_1)\cup V(H_2)$, $\mathcal{Q}[1,r+1]=\pi$ and $\mathcal{Q}[\ell-r, \ell]$ is a permutation of $V(H_{2})$.
\end{lemma}

\begin{proof}[Proof]
Let $V(H_{1})=\{v_{1}, v_{2}, \dots, v_{r+1}\}$ and $V(H_{2})=\{v_{3}, \dots, v_{r+1}, w_{1}, w_{2}\}$.
Let $\pi_{1}$ be any permutation of $V(H_{1})$.
We write $\pi_2=(v_{1}, v_{2}, \dots, v_{r+1})$ and $\pi_3=(v_{r+1}, w_{1}, w_{2}, v_{3}, \dots, v_{r})$.
Then Lemma~\ref{lem6.5} applied to $H_1$ gives a good walk say $\mathcal{Q}_1$ with $|\mathcal{Q}_1|\leq 65r^{3}$ which starts with $\pi_{1}$ and ends with $\pi_2$ and contains no lazy element among the last $20r+20$ elements.
Now we build
$\mathcal{Q}_2=\pi_2\pi_3\pi_3\cdots$ where we repeat $\pi_3$ 21 times to guarantee that there is no lazy element among the last $20r+20$ elements.
It is easy to verify that $\mathcal{Q}_2$ is a good walk with only one lazy element $S_{r+1}$.
Thus by piecing together $\mathcal{Q}_1,\mathcal{Q}_2$ and identifying the segment $\pi_2$, one can obtain a desired good walk of order less than $65r^{3}+22r+22\leq 100r^3$.
\end{proof}

Now we are ready to prove Lemma~\ref{lem6.2}.

\begin{proof}[Proof of Lemma \ref{lem6.2}]
Given $\mu>0$, $r\in\mathbb{N}$, we choose $\frac{1}{k}\ll\mu, \frac{1}{r}$.
Let $R$ be a $k$-vertex graph with $\delta(R)\geq \left(1-\frac{1}{r}+\mu\right)k$ and $\mf{x}$, $\mf{y}$ be two disjoint $r$-tuples of vertices in $R$, each inducing a copy of $K_{r}$, say $H_{\mf{x}}, H_{\mf{y}}$ with $V(H_{\mf{x}})=\{v_{1}, \dots, v_{r}\}$ and $V(H_{\mf{y}})=\{u_{1}, \dots, u_{r}\}$.
Without loss of generality, we write $\mf{x}=(v_{1}, \dots, v_{r})$ and $\mf{y}=(u_{r}, \dots, u_{1})$.

Applying Lemma \ref{lem6.3} to $R$, there exists a family of copies of $K_{r+1}$, say $\{H_{1}, \dots, H_{s}\}$ for some $s\leq r^{2}$, with $V(H_{\mf{x}})\< V(H_{1})$, $V(H_{\mf{y}})\< V(H_s)$ and $|V(H_{i})\cap V(H_{i+1})|=r-1$ for every $i\in[s-1]$.
Denote by $\mf{x}^+:=(v_{1}, \dots, v_{r}, v_{r+1})$, $\mf{y}^+:=(u_{r+1}, u_{r}, \dots, u_{1})$ the permutations of $V(H_1)$ and $V(H_s)$, respectively.

As mentioned above, by iteratively applying Lemma \ref{lem6.6} $s-1$ times, we obtain a collection of good walks $\mathcal{Q}_{i}$ of order $\ell_i\le 100r^3$, $i\in[s-1]$,
such that $\mathcal{Q}_{1}[1,r+1]=\mf{x}^+$ and $\mathcal{Q}_{i}[1,r+1]=\mathcal{Q}_{i-1}[\ell_{i-1}-r, \ell_{i-1}]$ when $i\ge 2$.
In particular, $V(\mathcal{Q}_{i})\subseteq V(H_{i})\cup V(H_{i+1})$ for every $i\in[s-1]$ and $\mathcal{Q}_{s-1}[\ell_{s-1}-r, \ell_{s-1}]$ is a permutation of $V(H_{s})$.
Furthermore, applying Lemma~\ref{lem6.5} to $H_s$ with $\pi_{1}=\mathcal{Q}_{s-1}[\ell_{s-1}-r, \ell_{s-1}]$, $\pi_{2}=\mf{y}^+$, we obtain a good walk, say $\mathcal{Q}_{s}$ with $|\mathcal{Q}_{s}|=:\ell_{s}\leq 65r^{3}$, such that $\mathcal{Q}_{s}[1,r+1]=\mathcal{Q}_{s-1}[\ell_{s-1}-r, \ell_{s-1}]$ and $\mathcal{Q}_{s}[\ell_{s}-r, \ell_{s}]=\mf{y}^+$.
Thus we can piece the walks $\mathcal{Q}_{1},\ldots,\mathcal{Q}_{s}$ together by identifying the $r+1$ coordinates from the ending of $\mathcal{Q}_{i}$ and the beginning of $\mathcal{Q}_{i+1}$ for all $i\in[s-1]$.
Then the resulting sequence $\mathcal{Q}$ is actually a good walk with head $\mf{x}$ and tail $\mf{y}$ as desired.
Moreover, it is easy to see that $|\mathcal{Q}|=\sum_{i\in[s]}\ell_i\leq 100sr^3\le 100r^5$ and Lemma \ref{lem6.5} guarantees that there is no lazy element among the last $20r+20$ elements of $\mathcal{Q}$. This completes the proof.
\end{proof}

\subsection{Proof of Lemma \ref{lem6.30} and Lemma \ref{lem6.31}}
In this subsection, we begin by proving Lemma \ref{lem6.31}. We then use Lemma \ref{lem6.31} to prove Lemma \ref{lem6.30}. Throughout the proofs, we will rely on the following trivial fact.

\begin{fac}[\cite{KS1996}]\label{fac5.10}
Given $\beta, \eps>0$, let $(A, B)$ be an $(\eps, \beta)$-regular pair, and $Y\subseteq B$ with $|Y|\geq \eps |B|$.
Then there exists a subset $A'\< A$ with $|A'|\leq \eps |A|$ such that every vertex $v$ in $A\backslash A'$ has $|N(v)\cap Y|\geq (\beta-\eps)|Y|$.
\end{fac}
In the following, we prove Lemma \ref{lem6.31}.

\begin{proof}[Proof of Lemma \ref{lem6.31}]
Given $\beta,\eta>0$, and $r, k, \ell\in \mathbb{N}$, we choose $\frac{1}{m}\ll \eps \ll \beta,\eta,\frac{1}{r}, \frac{1}{\ell}$.
Let $G$ be a graph with an equipartition $V(G)=V_{1}\cup V_{2}\cup \cdots\cup V_{k}$, and $|V_{i}|=m$ for every $i\in [k]$.
Let $R$ be a graph with $V(R)=\{V_{1}, V_{2}, \dots, V_{k}\}$, and $V_{i}V_{j}\in E(R)$ if $(V_{i}, V_{j})$ is $(\eps, \beta)$-regular in $G$.
Let $\mathcal{Q}=(S_{1}, \dots, S_{\ell+r})$ be a good walk in $R$ without any lazy element and $T_i\<S_i$ be a set of size at least $\eta m$ for every $i\in[\ell+r]$.
It suffices to prove that for every $s\in [\ell]$, there exists an $r$-path say $v_{1}v_2 \dots v_{s}$ such that $v_{i}\in T_{i}$ for every $i\in [s]$ and furthermore $|N_{G}(v_{s-t(s, j)+1}, \dots, v_{s})\cap T_{s+j}|\geq \left(\frac{\beta}{2}\right)^{t(s, j)}|T_{s+j}|$ for every $j\in [r]$, where $t(s, j)=\min\{s,r-j+1\}$.
We shall prove this by induction on $s$.

The base case $s=1$ is trivial.
Recall that $\mathcal{Q}$ has no lazy element and thus $(S_{1}, S_{j+1})$ is $(\eps, \beta)$-regular for every $j\in [r]$.
Note that $|T_i|\ge \eta m> r\eps m=r\eps |S_{i}|$ for every $i\in [\ell+r]$.
By Fact \ref{fac5.10}, there exists a subset $S'_{1}\< S_{1}$ with $|S'_{1}|\leq r\eps |S_{1}|$ such that every vertex $v\in T_{1}\backslash S'_{1}$ has $|N_{G}(v)\cap T_{j+1}|\geq (\beta-\eps)|T_{j+1}|$ for every $j\in [r]$.
As $|T_1|\ge \eta m> r\eps m=r\eps |S_{1}|\geq |S_1'|$, we choose an arbitrary vertex $v_{1}$ in $T_{1}\backslash S'_{1}$.
Then $|N_{G}(v_1)\cap T_{j+1}|\geq (\beta-\eps)|T_{j+1}|\geq \frac{\beta}{2}|T_{j+1}|$ for every $j\in [r]$.

Next, we show that our claim holds for $s+1$ assuming it holds for $s\ge 1$.
The induction hypothesis implies that there exists a vertex set $\{v_{1}, \dots, v_{s}\}$ such that $v_i\in T_i$ for $i\in[s]$ and the set $T_{s+j}^*:=N_{G}(v_{s-t(s, j)+1}, \dots, v_{s})\cap T_{s+j}$ has $|T_{s+j}^*|\ge \left(\frac{\beta}{2}\right)^{t(s, j)}|T_{s+j}|\geq \varepsilon |S_{s+j}|$ for every $j\in [r]$.
Recall that $(S_{s+1}, S_{s+1+j})$ is $(\eps, \beta)$-regular for every $j\in [r]$.
By Fact \ref{fac5.10}, there exists a subset $S'_{s+1}\<S_{s+1}$ with $|S'_{s+1}|\leq r\eps |S_{s+1}|$ such that every vertex $v\in T^{\ast}_{s+1}\setminus S'_{s+1}$ satisfies that for every $j\in [r-1]$ \[|N_{G}(v)\cap T^*_{s+1+j}|\geq (\beta-\eps)|T^*_{s+1+j}|\geq \frac{\beta}{2}\left(\frac{\beta}{2}\right)^{t(s,j+1)}|T_{s+1+j}|=\left(\frac{\beta}{2}\right)^{t(s+1,j)}|T_{s+1+j}|\] and $|N_{G}(v)\cap T_{s+1+r}|\ge (\beta-\eps)|T_{s+1+r}|\geq \frac{\beta}{2}|T_{s+1+r}|=\left(\frac{\beta}{2}\right)^{t(s+1,r)}|T_{s+1+r}|$.
As $\frac{1}{m}\ll\eps\ll\beta,\eta,\frac{1}{r},\frac{1}{\ell}$ and $s\leq \ell$, it holds that \[|T^{*}_{s+1}\backslash (S'_{s+1}\cup\{v_1,v_2,\ldots,v_s\})|\geq \left(\frac{\beta}{2}\right)^{r}|T_{s+1}|-r\eps|S_{s+1}|-s\geq\left(\frac{\beta}{2}\right)^{r}\eta m-r\eps m-s>0.\]
In $T^{*}_{s+1}\backslash (S'_{s+1}\cup\{v_1,v_2,\ldots,v_s\})$, we choose an arbitrary vertex $v_{s+1}$.
Then as $t(s+1,j)=t(s,j+1)+1$, it follows that \[|N_{G}(v_{(s+1)-t(s+1,j)+1},\ldots,v_s,v_{s+1})\cap T_{s+1+j}|=|N_{G}(v_{s+1})\cap T^*_{s+1+j}|\ge\left(\frac{\beta}{2}\right)^{t(s+1,j)}|T_{s+1+j}|\] for every $j\in [r]$, where $T^{*}_{s+1+r}=T_{s+1+r}$.
This finishes the proof.
\end{proof}

In the following, we use Lemma \ref{lem6.31} to prove Lemma \ref{lem6.30}.
\begin{proof}[Proof of Lemma \ref{lem6.30}]
Given $r, k\in \mathbb{N}$ and $\beta,\eta>0$, we choose $\frac{1}{m}\ll \alpha\ll \eps\ll\beta, \eta,\frac{1}{r}, \frac{1}{k}$.
Let $G$ be a graph with an equipartition $V(G)=V_{1}\cup V_{2}\cup\cdots\cup V_{k}$, $\alpha(G)\leq \alpha |V(G)|$, and $|V_{i}|=m$ for every $i\in [k]$.
Let $R$ be a graph with $V(R)=\{V_{1}, V_{2}, \dots, V_{k}\}$, and $V_{i}V_{j}\in E(R)$ if $(V_{i}, V_{j})$ is $(\eps, \beta)$-regular in $G$.
Let $\mathcal{Q}=(S_{1}, S_{2}, \dots, S_{3r+2})$ be a good walk in $R$ with exactly one lazy element $S_{r+1}$ and $T_i\<S_i$ be a set of size at least $\eta m$ for every $i\in[3r+2]$ such that $T_{r+1}=T_{r+2}$.
We first consider the subwalk $(S_{1}, S_{2}, \dots, S_{2r})$. Recall that $S_{r+1}$ is the lazy element.
Applying Lemma \ref{lem6.31} with $\ell=r$, there exists an $r$-path, say $P=v_{1} \dots v_{r}$, such that $v_{i}\in S_{i}$ for every $i\in[r]$ and $|N_{G}(v_{j}, \dots, v_{r})\cap T_{r+j}|\geq \left(\frac{\beta}{2}\right)^{r}|T_{r+j}|$ for every $j\in [r]$.
Define \[T^{0}_{r+j}:=N_{G}(v_{j}, \dots, v_{r})\cap T_{r+j}~\text{for every}~j\in [r]\] and $T^{0}_{r+j}:=T_{r+j}$ for every $j\in [r+1, 2r+2]$.
To finish the proof, it suffices to prove the following claim.

\begin{claim}\label{cl5.11}
There exists a vertex set $\{v_{r+3}, \dots, v_{2r+2}\}$ disjoint from $\{v_1,\ldots,v_r\}$ such that $v_{i}\in T^{0}_{i}$ for every $i\in [r+3,2r+2]$ and $\{v_{r+3}, \dots, v_{2r+2}\}$ induces a clique.
Moreover, there exists a sequence of subsets $T^{1}_{r+1}:= N_{G}(v_{r+3}, \dots, v_{2r+2})\cap T^0_{r+1}$ and\[T^{1}_{j}:=N_{G}(v_{j-r}, \dots, v_{2r+2})\cap T^0_{j}~\text{for every}~j\in[2r+3, 3r+2]\] such that $|T^{1}_{j}|\geq \left(\frac{\beta}{2}\right)^{r}| T^0_{j}|$ for every $j\in[2r+3, 3r+2]\cup \{r+1\}$.
\end{claim}

\begin{pr}
The proof of the claim follows from the same argument as in that of Lemma~\ref{lem6.31}, except an additional requirement on the choice of $v_{r+3},\ldots,v_{2r+2}$ that $|T^{1}_{r+1}|\geq \left(\frac{\beta}{2}\right)^{r}| T^0_{r+1}|$, and we omit it.
\end{pr}

By Claim \ref{cl5.11}, there exists a subset $T^{1}_{r+1}\< N_{G}(v_{1}, \dots, v_{r},v_{r+3}, \dots, v_{2r+2})$ with $|T^{1}_{r+1}|\geq \left(\frac{\beta}{2}\right)^{r}|T^0_{r+1}|\ge \left(\frac{\beta}{2}\right)^{2r}\eta m>\alpha km=\alpha|V(G)|$, as $\alpha\ll \beta,\eta,\frac{1}{r},\frac{1}{k}$.
Hence, $G[T^{1}_{r+1}]$ contains an edge, say $xy$.
Then the $r$-path $P=v_{1}\dots v_{r}xyv_{r+3}, \dots, v_{2r+2}$ is as desired.
\end{proof}

\bibliographystyle{abbrv}
\bibliography{ref}

\end{document}